\documentclass[a4paper,11 pt]{amsart}

\usepackage{amssymb}  
\usepackage{mathtools}      
\usepackage{mathabx}        
\usepackage[bb=fourier,cal=euler,scr=rsfs]{mathalfa}	
\usepackage{enumitem}       
\usepackage[colorlinks=true,backref=page]{hyperref} 

\usepackage[ansinew]{inputenc}

\renewcommand*{\backref}[1]{}
\renewcommand*{\backrefalt}[4]{\quad \tiny
  \ifcase #1 (\textbf{NOT CITED.})%
  \or    (Cited on page~#2.)%
  \else   (Cited on pages~#2.)%
  \fi}

\hypersetup{bookmarksdepth = 3} 
\numberwithin{equation}{section}     

\setlist[enumerate,1]{label={\upshape(\roman*)},ref=\roman*}
\setlist[enumerate,2]{label={\upshape(\alph*)},ref=\alph*}

\newcommand{\Z}{\mbox{$\mathbb{Z}$}}

\newcommand{\R}{\mbox{$\mathbb{R}$}}

\newcommand{\T}{\mbox{$\mathbb{T}$}}

   \def\DD{{\mathbb D}}

 \def\RR{{\mathbb R}}  \def\TT{{\mathbb T}}
   
 \def\ZZ{{\mathbb Z}}

\def\cA{\mathcal{A}}  \def\cG{\mathcal{G}}  \def\cS{\mathcal{S}}
\def\cB{\mathcal{B}}    \def\cT{\mathcal{T}}
  \def\cI{\mathcal{I}}  
\def\cD{\mathcal{D}}    
    
\def\cF{\mathcal{F}}  \def\cL{\mathcal{L}}

\newtheorem*{teo*}{Theorem}

\newtheorem{teo}{Theorem}[section]

\newtheorem{quest}{Question}

\newtheorem{claim}[teo]{Claim}
\newtheorem{lema}[teo]{Lemma}
\newtheorem{lemma}[teo]{Lemma}
\newtheorem{prop}[teo]{Proposition}
\newtheorem{proposition}[teo]{Proposition}
\newtheorem{definition}[teo]{Definition}

\theoremstyle{definition}

\theoremstyle{remark}

\newtheorem{remark}[teo]{Remark}

\newcommand{\eps}{\varepsilon}

\newcommand{\mt}{\widetilde{M}}

\newcommand{\wt}[1]{\widetilde{#1}}
\newcommand{\Ft}{\widetilde{\cF}}
\newcommand{\tild}[1]{\widetilde{#1}}
\newcommand{\suniv}{S^1_{univ}}
\newcommand{\tj}{S^1_{JSJ}}

\title[Minimality in the universal circle]{Minimality of the action on the universal circle of uniform foliations}

\author[S.~Fenley]{Sergio R.\ Fenley} 
\address{Florida State University, Tallahassee, FL 32306}
\email{fenley@math.fsu.edu}

\author[R.~Potrie]{Rafael Potrie} 
\address{Centro de Matem\'atica, Universidad de la Rep\'ublica, Uruguay}
\email{rpotrie@cmat.edu.uy}
\urladdr{http://www.cmat.edu.uy/~rpotrie/}

\begin{document}

\begin{abstract}
Given a uniform foliation by Gromov
hyperbolic leaves on a $3$-manifold, we show that the action of the fundamental group on the universal circle is minimal and transitive on pairs of different points. 
We also prove two other results: we prove that general uniform Reebless foliations are $\R$-covered and we give a new description of the universal circle of $\R$-covered foliations with Gromov hyperbolic
leaves in terms of the JSJ decomposition of $M$. 
\end{abstract}

\maketitle

\section{Introduction}
Consider a Reebless foliation $\cF$ on a closed 3-manifold $M$ without 
spherical or projective plane leaves. This implies that the universal cover $\mt$ of $M$ is homeomorphic\footnote{We note that we will always work with $\mt$ as a Riemannian manifold where the metric is induced by lifting the metric of $M$ to the universal cover. As such, the manifold $\mt$ may be very different from $\mathbb{R}^3$ even if homeomorphic.} to $\mathbb{R}^3$ 
\cite{Palmeira} and that every leaf of $\cF$ is a properly embedded plane
\cite{Nov}. We denote by $\Ft$ to the lift of $\cF$ to $\mt$. 

A specific class of foliations are those called \emph{uniform} which means that in the universal cover, any two leaves are at finite Hausdorff distance from each other. 
See section \ref{ss.uniffol} for the several variations 
of the definition of uniform foliations.
Fibrations over the circle are one
obvious example. Much more generally, slitherings, introduced 
by Thurston in \cite{Thurston} (see also \cite{Calegari}) are examples of such foliations. 
In addition from any slithering
example one can construct other examples of uniform
foliations by blowing up some leaves into 
foliated interval bundles.
All of these examples of uniform foliations are what is
called $\R$-covered. 
Recall that a foliation is $\R$-covered if the \emph{leaf space} $\cL_\cF = \mt/_{\Ft}$ of $\Ft$ is homeomorphic to $\R$. 
We first prove:

\begin{teo}\label{teo-Rcovered}
A uniform Reebless foliation in a closed 
$3$-manifold $M$ is $\R$-covered. 
\end{teo}

This result has no restriction on the intrinsic metric in the
leaves. 

Theorem \ref{teo-Rcovered} implies 
that all Reebless uniform foliations are obtained
from either slithering foliations or 
blow ups of slithering foliations as explained in \cite[Construction 9.14 and Theorem 9.15]{Calegari}
using a result of Thurston \cite[Theorem 2.7]{Thurston}.
The proof implicitly uses the fact that the foliation
is $\R$-covered. We provide a proof of the $\R$-covered
property here.

The requirement of Reebless in Theorem \ref{teo-Rcovered}
is not superfluous:
any foliation in the $3$-sphere $S^3$ (or any closed $M^3$ with
finite fundamental group) is uniform,
however none are $\R$-covered, because they have Reeb components. 
To prove Theorem \ref{teo-Rcovered} it is enough to show that the leaf space 
is Hausdorff (see e.g. \cite{Calegari, CanCon}, a quick account of the most relevant material is presented in \S \ref{ss.taut}). 

Some uniform foliations are quite special, for example linear foliations in $\mathbb{T}^3$ or in nilmanifolds. These foliations have leaves
that are parabolic. But for most uniform foliations, one can apply a beautiful result of Candel \cite{Candel} to see that there is a metric on $M$ making each leaf negatively curved (see e.g. \cite[\S 5.1]{FP-acc} for a specific statement). 

For our next result 
we will consider the following setting: $\cF$ will be a uniform foliation on a closed Riemannian 3-manifold $M$ such that the metric restricted to each leaf of $\cF$ is Gromov hyperbolic (in particular, it has to be Reebless since the torus does not admit a Gromov hyperbolic metric). We will call such foliations \emph{uniform hyperbolic} foliations. 

For such foliations, one can consider, for each leaf $L \in \Ft$ the circle at infinity $S^1(L)$ defined as the set of geodesic rays up to being a finite Hausdorff distance apart (see \S \ref{ss.boundaries}). The fact that the foliation is $\R$-covered is very useful to define a \emph{universal circle} $\suniv$ which is essentially a canonical way to identify all the $S^1(L)$ as one varies $L \in \Ft$. The precise definition will be given in \S \ref{ss.universal}. See also \cite{Thurston, Calegari, CalegariPA, Fen2002, CD, Frankel} among other places where universal circles are defined in even more general situations.  

Our main result is the following: 

\begin{teo}\label{teo-Minimal}
Let $\cF$ be a uniform hyperbolic foliation on a 3-manifold $M$. Then, the fundamental group $\pi_1(M)$ acts minimally on the universal circle $\suniv$. Moreover, the diagonal action on pairs of different points of $\suniv$ has dense orbits. 
\end{teo}

This result extends a very well known result about actions of hyperbolic groups on their Gromov boundary (see \cite[\S 8.2]{Gromov}) and complements well with \cite[Lemma 5.2.2]{CalegariPA} which is stated for non-uniform $\R$-covered foliations. Note that in the case where the foliation is a fibration this follows from the corresponding result for fundamental group actions of surfaces in their boundary.
For Anosov foliations the result is also easily proved using the following: the flow is $\R$-covered and since the foliation is uniform,
the flow is skewed \cite{FenleyAnosov}. The structure of skewed
Anosov flows is very rich and well understood \cite{FenleyAnosov,Thurston}: the second statement of theorem \ref{teo-Minimal} follows from
the existence of a dense orbit of the flow. The first statement
follows from the minimality of the Anosov foliation \cite{FenleyAnosov}
and the structure of the flow.

Theorem \ref{teo-Minimal}
was motivated by some applications to partially hyperbolic dynamics (it will be used in \cite{FP-2}). We hope this result may have independent interest or find other applications.

Some proofs of intermediate steps are simpler if one restricts to the case of atoroidal 3-manifolds where one has transverse pseudo-Anosov flows that helps understanding the action on the universal circle (\cite{Thurston,CalegariPA,Fen2002}).

When the manifold has a non-trivial JSJ decomposition, the proof includes a careful study of the intersection between leaves of the foliation and the pieces of the JSJ decomposition. This results in a new way to look at the universal circle that may be of independent interest and holds for general 
(both uniform and non uniform) $\R$-covered foliations. See Proposition \ref{prop-JSJvsUniversalCircle}. 

Because of our applications, at the end of the paper we explain how the results hold also for \emph{branching foliations}, which are a technical object featuring often in partially hyperbolic dynamics.

\section{Preliminaries} 

\subsection{Reebless foliations}\label{ss.taut} 
We will be mainly concerned with Reebless foliations in this article.
See \cite[\S 4]{Calegari} for a broad introduction. 

A Reeb component is a foliation of the solid torus, such that the
boundary is a leaf. In addition all the leaves in the interior are
planes and spiral or limit towards the boundary. There is a 
circle worth of leaves in the interior.
By an abuse of terminology we also consider Reeb component
a quotient of this, which may be a foliation of a solid Klein
bottle.
If a foliation by surfaces $\cF$ in a closed 3-manifold $M$ does not have \emph{Reeb components} it follows from a celebrated result of Novikov \cite{Nov} that when lifted to the universal cover, the foliation is made of simply connected leaves and the \emph{leaf space} $\cL_\cF = \mt /_{\Ft}$ is a simply connected (possibly non-Hausdorff) one-dimensional manifold. If there is a leaf of $\Ft$ which is a sphere or a projective plane, it follows that the foliation $\Ft$ is equivalent to the trivial foliation by spheres in $S^2 \times \R$. 
If there are no projective space or spherical leaves of $\cF$ then a result of Palmeira \cite{Palmeira} implies that $\mt$ is homeomorphic to $\R^3$.  We refer the reader to \cite{CanCon,Calegari} for a broad treatment, we will assume some familiarity with the theory of foliations. 

We will not be too precise about regularity of our foliations. Everything works for foliations of class $C^{1,0+}$ as defined in \cite{CanCon} (i.e. continuous with $C^1$ leaves tangent to a continuous distribution). Thanks to \cite{Calegari0} in view of the nature of our result, this is a quite general assumption. 

To show that a foliation is $\R$-covered, it is enough to show that its leaf space is Hausdorff (see e.g. \cite[Lemma 2.2]{Fen2002}). 

A \emph{taut foliation} is a foliation such that every leaf intersects
a closed transversal. Notice that taut foliations must be Reebless\footnote{There is a subtlety in the definition of tautness for $C^{1,0+}$ foliations. Here we will keep the definition we made which does not change our results. See \cite{CKR}.} 

Another relevant result about foliations in 3-manifolds is the following (see \cite{Gabai} or \cite[Theorem II.9.5.5]{CanCon}):

\begin{teo}[Roussarie-Gabai]\label{teo-transversetori}
Let $\cF$ be a taut foliation in a 3-manifold $M$ and let $T \subset M$ be
an embedded incompressible torus or Klein bottle.
Then, $T$ can be isotoped to be either 
a leaf of $\cF$ or in general
position with respect to $\cF$. In particular in the second
case the induced foliation
by $\cF$ in $T$ does not have singularities.
If $\cF$ is taut one can isotope $T$ to be 
either a leaf of $\cF$ or transverse to $\cF$. 
\end{teo}

\subsection{Uniform foliations}\label{ss.uniffol}
In this paper we will mainly concentrate in the following class of foliations. 

\begin{definition} \label{d.uniform}
Let $\cF$ be a foliation in a manifold $M$.
We say that $\cF$ is uniform, if for any two leaves
$L, F$ of the lifted foliation $\widetilde \cF$ to $\mt$,
then the Hausdorff distance between $L$ and $F$ is finite.
\end{definition}

There have been several forms of the definition of uniform
foliations, which we review here. Our definition is the weakest
or most general possible. In his seminal article \cite[Definition 2.1]{Thurston},
Thurston originally defined uniform foliation 
as a codimension one foliation in any
dimension satisfying Definition \ref{d.uniform}
and such that in addition any closed transversal is not null homotopic.
Calegari \cite[Definition 2.1.5]{CalegariPA} or \cite[Definition 9.13]{Calegari}, defined uniform for codimension one
foliations in $3$-manifolds $M$ satisfying
Definition \ref{d.uniform} and so that the foliation is also taut.
The first author \cite[Definition 2.4]{Fen2002} defined 
uniform for codimension foliations in $3$-manifolds satisfying Definition \ref{d.uniform}. 

Note that Definition~\ref{d.uniform} does not require $M$ to be $3$-dimensional
or $\cF$ codimension one, but we will restrict to this case in this paper. 

Thurston \cite{Thurston} remarks on the connection of the uniform property
with the Reebless condition for codimension one foliations
in $3$-manifolds. After \cite[Definition 2.1]{Thurston} it is stated that if a foliation verifies that every closed transversal is not-nullhomotopic then there are no Reeb components. This is true if one additionally assumes that the foliation is uniform, and we prove this in \S~\ref{ss.reeb}.

\subsection{JSJ decompositions}\label{ss.3mfds}
We refer the reader to \cite{Hatcher} for a more complete account on this. 

What we will use is that every irreducible\footnote{Note that if there
is a Reebless foliation in $M$, then $\widetilde M$ is homeomorphic
to $\R^3$, hence $M$ is irreducible.} 3-manifold $M$  admits a cannonical collection (unique up to isotopy) of embedded incompressible
tori and Klein bottles 
 $T_1, \ldots, T_k$ such that if we cut $M$ along the tori$/$Klein bottles, each \emph{piece} (i.e. connected component of the complement) is either atoroidal or Seifert. We will exclude the case where $M$ is a torus bundle up to finite cover since in this case there can be a unique piece which is Seifert but its fibration may not be not unique up to isotopy $-$ for
example in $\T^3$. When $k \geq 1$ and $M$ is not a torus bundle up to finite cover, we say that $M$ has a \emph{non-trivial JSJ decomposition}.

\begin{remark} We will abuse terminology use and refer to 
$T_1, \ldots, T_k$ as the JSJ tori, even though some components
may be Klein bottles.
\end{remark}

Let $M$ be an irreducible 3-manifold with non-trivial JSJ decomposition and let $M_1, \ldots, M_n$ the pieces of its JSJ decomposition (i.e. the connected components of $M \setminus \{T_1, \ldots, T_k\}$, notice that it could be that $n=1$ even if $k\geq 1$). In $\mt$ we consider all connected components of the lifts $\tild{M}_i^j$ of each $M_i$.

It turns out that the following holds:

\begin{prop}\label{p.JSJtree}
The graph consisting of vertices in each of the $\tild{M}_i^j$ and edges between vertices sharing a boundary is an infinite tree $\cT$. Moreover, the fundamental group acts naturally on $\cT$ and for every element $\gamma  \in \pi_1(M)$ the set of fixed points of $\gamma$
in $\cT$ has diameter\footnote{We are using the standard metric on a graph making each edge have length equal to $1$.} at most $2$. 
\end{prop}

\begin{proof}
The fact that it is a tree follows directly because the lift of a JSJ
torus to the universal cover is a properly embedded plane which separates $\mt$ in exactly two connected components and this forbids the graph to have closed loops. This also implies that each $M_i$ has infinitely many lifts; to see this, notice that if a boundary torus of $M_i$ has infinitely many lifts in some $\tild{M}_i^{j_0}$ then clearly there must be infinitely many $\tild{M}_i^j$ because there should be at least one lift of $M_i$ in each complementary region of the torus in $\mt$ not containing $\tild{M}_i^{j_0}$. If all boundary torus have finitely many lifts, then the deck transformations that fix $\tild{M}_i^{j_0}$ are a finite extension of $\Z \oplus \Z$ so there must be infinitely many deck transformations moving $\tild{M}_{i}^{j_0}$ pairwise disjointly (notice that finite extensions of $\Z \oplus \Z$ cannot be the fundamental group of a closed, irreducible 3-manifold by homological reasons). Indeed, this implies that each $\tild{M}_i^j$ has infinitely many boundary components (this uses the fact that $M_i$ cannot be $\mathbb{T}^2 \times (0,1)$ as our definition of having non-trivial JSJ decomposition which expressly excludes the case of torus bundles).   

The fact that the set of fixed points of a deck transformation has diameter at most $2$ in $\cT$ follows the strategy the proof of \cite[Lemma A.1]{BFFP2}. We sketch the main points for completeness. As in \cite[Appendix A]{BFFP2} we will call the components of the lifts of tori in the JSJ collection \emph{walls}. 

We let $\gamma \in \pi_1(M)$ be a deck transformation.  We first notice that if $M_i$ is an atoroidal piece then if $\tild{M}_i^j$ is fixed by $\gamma$ then at most one wall of $\tild{M}_i^j$ can be invariant under $\gamma$.
Otherwise one gets a $\pi_1$-injective annulus in $M_i$ with boundary
in boundary of $M_i$ which is not homotopic rel boundary to boundary of 
$M_i$. Using this annulus and annuli in boundary components of $M_i$
one can piece together a $\pi_1$-injective torus or Klein bottle
in $M_i$ which is
not homotopic to the boundary, contradicting that $M_i$ is
atoroidal.

Now, if $M_i$ is a Seifert piece, then we claim that if $\tild{M}_i^j$ is fixed by $\gamma$ then by a similar argument we see that if more than one wall is fixed, then $\gamma$ must belong to the center of $\pi_1(M_i)$ (i.e. the element generated by the fibers of the Seifert fibering) in which case, $\gamma$ cannot belong to the center of the Seifert pieces that are adjacent to $\tild{M}_i^j$. 

This shows that any connected component of the fixed point set of $\gamma$ has diameter at most two. But since $\cT$ is a tree and $\gamma$ acts by isometries, the fixed point set is connected. This concludes. 
\end{proof}

\subsection{Boundaries at infinity}\label{ss.boundaries}
Let $X$ be a negatively curved complete space with curvature bounded from below and above. See \cite{Gromov, Ledrappier} for general references. 

For such a space we define a boundary at infinity $\partial_\infty X$ defined as the equivalence relation of geodesic rays up to being at a bounded distance (see \cite[\S I]{Ledrappier}). When $X$ is a surface, the negative curvature implies that if $X$ is simply connected then it is homeomorphic to $\mathbb{D}^2$ and one can identify the boundary $\partial_\infty X$ with the circle of directions $T^1_x X$ at any point $x\in X$. So, for simply connected surfaces of negative curvature, we denote the boundary at infinity as $S^1(X)=\partial_\infty X$. 

The metric in $S^1(X)$ is only well defined up to H\"{o}lder equivalence since it is intended to be an invariant under quasi-isometries. For our purposes, it will be convenient to choose a special metric on $S^1(X)$ called the \emph{visual metric}. For this, we fix a point $x_0 \in X$ and we measure the length of an interval $I \subset S^1(X)$ by looking at the angle formed by the interval in $T^1_{x_0} X$ of vectors whose geodesic ray starting at $x_0$ lands in a point of $I$. 
The visual measure is the Lebesgue measure induced by this metric.
This is clearly dependent on the point, but we will always explicit the point we are considering.

\subsection{Universal circles}\label{ss.universal}
In this section  we will review the construction of the
universal circle for an
$\R$-covered foliation $\cF$ on a closed 3-manifold $M$ so that it admits a metric which restricts in each leaf to a negatively curved surface with curvature\footnote{For foliations, \cite{Candel} provides a metric of curvature exactly $-1$, but since we want to apply this result in a slightly more general case (that is of branching foliations), we only use that the curvature is uniformly close to $-1$. Notice also that the metric constructed by \cite{Candel} may be only $C^0$ transversally to leaves, and since we are concerned only with quasi-isometric properties of leaves, it is more than fine to have just negative curvature or CAT(-1) leaves. See \cite[\S A.3]{BFP} for a more complete account.} close to $-1$.  

Denote by $\Ft$ the lift of $\cF$ to $\mt$, the universal cover of $M$. For each $L \in \Ft$ we define $S^1(L)$ to be the boundary at infinity of $L$, which is well defined thanks to the fact that $L$ is negatively curved. 
First there is the cylinder at infinity $\cA$ which is the union
of the  $S^1(L)$ for $L$ leaf in $\widetilde \cF$.
The topology in $\cA$ is given by: given $x$ in $\widetilde M$,
let $\tau$ be a small transversal to $\widetilde \cF$ through $x$.
For every point $y$ of $\tau$, $y$ is in $L \in \widetilde F$.
For every $v$ in the unit tangent bundle of $L$ at $y$, let
$\gamma_v$ be the geodesic ray starting at $y$ with direction $v$.
The ideal point $z_v$ of $\gamma_v$ is a point in $S^1(L)$. 
It is well known that 
since $L$ has negative curvature, the map $v \rightarrow z_v$
is a homeomorphism. In the same way one defines a map

$$\eta: T^1 \widetilde \cF | \tau \ \ \rightarrow \ \ 
B_{\tau} \  = \  \bigcup_{L \cap \tau \not = \emptyset} S^1(L)$$

\noindent
which is the map $v \rightarrow z_v$ for any $y$ in $\tau$.
Put a topology in $B_{\tau}$ so that this map is a homeomorphism.
Do this for a $\pi_1(M)$ invariant collection of transversals
with union intersecting every $L \in \widetilde \cF$.
In \cite{Fen2002} it is proved that the topology in the intersection
of subsets of $\cA$ is well defined. This makes $\cA$ into an
open annulus, and $\pi_1(M)$ acts by homeomorphisms on this.
In addition there is a topology on $\widetilde M \cup \cA$
making it homeomorphic to $\DD^2 \times \RR$ and so that
each $L \cup S^1(L)$ corresponds to $\DD^2 \times \{ t \}$ 
for some $t$. Again $\pi_1(M)$ acts by homeomorphisms on
this topology.
We now describe the universal circle of $\cF$.

\subsubsection{ Case of $\cF$ uniform.}\label{sss.uniform}

We denote, for $L, F \in \Ft$ a map $\tau_{L,F}: L \cup S^1(L) \to F \cup S^1(F)$ which has the following properties: 

\begin{itemize}
\item $\tau_{L,F}|_L$ is a quasi-isometry with constant $c>1$ depending only on the Hausdorff distance between $L$ and $F$,
\item $\tau_{L,F}|_{S^1(L)}$ is a homeomorphism, 
\item $\tau_{F,G} \circ \tau_{L,F}|_{S^1(L)}= \tau_{L,G}|_{S^1(L)}$. 
\end{itemize}

See \cite[\S 5]{Thurston} or \cite[Corollary 5.3.16]{CalegariPA} or \cite[Proposition 3.4]{Fen2002}. 
Roughly the construction of such a collection of maps $\tau_{L,F}$ is
as follows:
Recall that a \emph{quasi-isometry} of constant $c>1$ is a map $\phi: L \to F$ so that $c^{-1} d_{L}(x,y) - c < d_F(\phi(x), \phi(y)) < c d_{L}(x,y) + c.$
Given $L, F$, the Hausdorff distance between them is $a_0 > 0$.
Given any $x$ in $F$ there is $y$ in $L$ with $d(x,y) < a_0 + 1$.
Let $\tau_{L,F}(x) = y$. This map is well defined up to an error $a_1$, with $a_1$ depending only on
$a_0$, see \cite[\S 3]{Fen2002}. 

The map $\tau_{L,F}$ is
a quasi-isometry, so it extends
to $L \cup S^1(L)$, and it is a homemorphism into
its image restricted to $S^1(L)$.

Recall that a quasigeodesic is a quasi-isometry from $\Z$ or $\R$
into $\widetilde M$. Since the map $\tau_{L,F}|_L$ is a quasi-isometry
it takes quasigeodesics in $F$ to quasigeodesics in $L$.
It is easy to see that $x \in S^1(L) \sim y \in S^1(F)$ if and only
if a quasigeodesic $\alpha$ in $L$ with ideal point $x$ is a
finite Hausdorff distance from a quasigeodesic in $F$ with ideal
point $y$.

The universal circle of $\cF$ is then defined as the circle $\suniv$ which is 
$\cA/_{\sim}$ where $x \in S^1(L) \sim y \in S^1(F)$ if $y = \tau_{L,F}(x)$. Notice that it is easy to see that $\suniv$ can be identified with $S^1(L)$ for every $L \in \Ft$, so one can think of $\suniv$ as a cannonical way to identify all boundaries at infinity of leaves.

\vskip .1in
The fundamental group $\pi_1(M)$  acts 
on $\suniv$ by homeomorphisms. This is because
any $\gamma$ in $\pi_1(M)$ sends pairs of quasigeodesics
in leaves which are a finite Hausdorff distance apart to
like pairs in $\gamma(L), \gamma(F)$.

\begin{remark}\label{rem-action}
Let $\gamma$ in $\pi_1(M)$ and $L$ a leaf of $\Ft$.
The action of $\gamma \in \pi_1(M)$ on $\suniv$ can be represented
by an action on $S^1(L)$ identifying $S^1(L) \cong \suniv$ and so the action is obtained by the deck transformation composed with $\tau_{\gamma L, L}$. We denote the action of $\gamma$ on $S^1(L)$ obtained as $\tau_{\gamma L, L} \circ \gamma$ by $\rho(\gamma)$. 
\end{remark}

\subsubsection{ Case of $\cF$ not uniform.}\label{sss.notuniform}

We refer to \cite{Fen2002} (see also \cite{CalegariPA}). In this case there are no
compact leaves of $\cF$ \cite[Lemma 2.5]{Fen2002},
and $\cF$ has a unique minimal
set $\cL$ \cite[Proposition 2.6]{Fen2002}.
Each complementary component of $\cL$ is a
$(0,1)$-bundle and $\cF$ can be collapsed to produce
a minimal foliation \cite[Proposition 2.6]{Fen2002}.
Hence
one can assume that $\cF$ is minimal. There is also a
canonical collapsing between the cylinders at infinity.

So assume that $\cF$ is minimal. In \cite[\S 3]{Fen2002} it is 
proved that for any $L, F$ in $\widetilde \cF$ there is a
dense set of directions between them which is a 
contracting direction between them. This means the following:
Fix $x$ in $L$.
There is a dense set of points $B$ in $S^1(L)$ so that
for any $y$ in $B$ if $\gamma$ is a geodesic ray in $L$ 
starting in $L$ and with ideal point $y$, then $\gamma$ is 
asymptotic to $F$ (and hence to any leaf in between $L, F$).
Asymptotic means that distance between $\gamma$ and $F$
goes to $0$ as points escape in $\gamma$.
For any $E$ between $L, F$ there is a geodesic ray in $E$
asymptotic to $\gamma$. This defines an ideal point in $S^1(E)$.
The union of these ideal points over such $E$ is a continuous
curve in $\cA$. The union of these for all $y$ in $B$ is 
a dense set in the subset $\cD$ of $\cA$ between $S^1(L)$ and $S^1(L)$.
This extends uniquely to a foliation in $\cD$ by intervals,
each interval intersects a circle at infinity once and only once.
One iterates this procedure making $L, F$ escape compact sets
of the leaf space in opposite directions. This defines a foliation
in $\cA$ by vertical lines, each intersecting a circle at infinity
once and only once.

The universal circle of $\cF$ is the quotient $\cA/\sim$ where $~$ is
the equivalence relation of being in the same leaf of the
vertical foliation.
The group of deck transformations $\pi_1(M)$ acts by 
homeomorphisms preserving the vertical foliation in $\cA$. 
This is because it sends the contracting directions as above
to contracting directions.
Hence $\pi_1(M)$ acts
by homeomorphisms on the universal circle $\suniv$.
Both the vertical foliation and the universal circle
pull back to the original foliation before collapsing
complementary regions of the minimal set $\cL$.

\vskip .1in
The existence of a the universal circle is much more general.
It exists for every foliation with Gromov hyperbolic leaves \cite{CD}.
In addition 
in \cite{CD} a universal circle is constructed for every 
tight essential lamination.
See \cite{Calegari} for more on this theory.

\section{Uniform foliations: proof of Theorem \ref{teo-Rcovered}}\label{s.uniform}
In this section we prove Theorem \ref{teo-Rcovered}. We first discuss in \S \ref{ss.reeb} the Reebless assumption. This subsection is independent of the proof of Theorem \ref{teo-Rcovered} and can be safely skipped. The results in \S \ref{ss.bddist} hold in more generality than the case of uniform foliations and could be of independent interest. 

As explained in \S \ref{ss.taut}, Theorem \ref{teo-Rcovered} is immediate if the foliation has spherical or projective plane leaves by the Reeb stability theorem which implies in that case that up to finite cover the foliation is the trivial foliation by spheres in $S^2 \times S^1$. So in this section {\bf we will assume throughout that leaves of $\cF$ are not spheres or projective planes.}

\subsection{Some remarks on the Reebless assumption}\label{ss.reeb}

It can certainly be the case that a foliation with Reeb components is uniform yet not $\R$-covered. Indeed, if $M$ has finite fundamental group, any foliation in $M$ has Reeb components by Novikov's theorem \cite{Nov} while the universal cover is compact, so the foliation is uniform. Notice that a Reeb component has non-Hausdorff leaf space: every neighborhood of the boundary leaf contains all the leaves of the interior of the solid torus as these all accumulate the boundary. Foliations of closed 3-manifolds with finite fundamental group are all examples of uniform non-$\R$-covered foliations:  

\begin{quest}\label{p.unifimpliesReebless}
If $\cF$ is uniform in $M$ with infinite fundamental group, does it
follow that $\cF$ is Reebless? 
\end{quest}

We don't know how to prove this in all generality, however we can
prove the following intermediate fact.

\begin{lema}\label{l.compactclosure}
If $\cF$ a foliation in $M$ has a Reeb component and it is uniform then every leaf in the universal cover has compact closure. In particular, any non-torsion element $\gamma \in \pi_1(M)$ acts freely on the leaf space $\cL_\cF = \mt/_{\Ft}$. 
\end{lema}

\begin{proof}
Let us first assume that the fundamental group of the boundary tori of the Reeb component does not map to $0$ in the fundamental group $\pi_1(M)$ of $M$. This implies that the Reeb torus lifts to a Reeb cylinder where leaves accumulate on one end of the cylinder. 
Let $\gamma$ represent the deck transformation associated with 
the core of the Reeb component.
Assume that the basepoint $p$ is in a lift $\tilde \gamma$.
Then in $\widetilde M$, $d(\gamma^n p, p) \rightarrow \infty$.
One can see this in the Cayley graph of $\pi_1(M)$ with an
edge metric. 
The Cayley graph is quasi-isometric to the universal cover $\widetilde M$.
In the Cayley graph there are finitely many elements in the 
ball of any radius and this implies that 
$d(\gamma^n p, p) \rightarrow \infty$.
In particular this implies that the 
leaves inside of the cylinder are not a bounded distance away from the cylinder, so the foliation is not uniform. 

Now, assume that the Reeb component lifts to $\mt$, so there are compact leaves of $\Ft$. It follows that every leaf $L \in \Ft$ is a finite 
Hausdorff distance from a compact leaf. In particular
$L$ is bounded, therefore its closure is compact. 
\end{proof}

In particular this implies that if $\cF$ is uniform and
every closed transversal is not null homotopic then $\cF$
is Reebless, as announced in subsection  \ref{ss.uniffol}.

\subsection{Lifts of leaves at a bounded distance}\label{ss.bddist}

In this section we prove some general results about Reebless foliations. We will use them in the next subsection to prove Theorem \ref{teo-Rcovered}. 

Recall that the leaf space $\cL_\cF = \mt/_{\Ft}$ in this case is a simply connected 1-dimensional manifold which is possibly non-Hausdorff. This is because for every leaf $L \in \Ft$ if $t$ is a transversal (i.e. a curve transverse to $\Ft$ homeomorphic to an open interval and intersecting $L$) it holds that $t$ intersects each leaf of $\Ft$ at most once (cf. \S \ref{ss.taut}). 

We note here that a natural idea would be to use the Hausdorff distance between leaves in the universal cover to show that the leaf space is Hausdorff. For a Reebless, uniform foliation, leaves in $\mt$ separate and hence the Hausdorff distance between
leaves induces a metric in the leaf space, which we can call 
the Hausdorff metric. 
However in general this metric induces a topology which is 
completely different from the quotient topology in the leaf space due to lack of compactness.
Consider for example an Anosov flow which is $\R$-covered
but not topologically conjugate to a suspension.
It follows that if $\cF$ is the weak stable foliation of this flow,
then $\cF$ is also uniform \cite{Thurston}.
Transversely to this foliation there is a strong unstable foliation
and using that it is very easy to see that there is $a_0 > 0$ such that
for any two distinct leaves $L, E$ of $\widetilde \cF$,  then
the Hausdorff distance between $L$ and $E$ is finite but bigger
than $a_0$. In other words the Hausdorff metric induces the
discrete topology
in the leaf space. This is completely different from the quotient
topology making it homeomorphic to the reals.
Notice that it is easy to see that the topology induced by
the Hausdorff metric is always bigger than the quotient 
topology.

As explained in \S \ref{ss.taut}, when $\cF$ is not $\R$-covered, there 
are
\emph{non-separated} leaves of $\Ft$: that is, leaves $L,F \in \Ft$ so that for every transversals $t_L, t_F$ to respectively $L$ and $F$ one has a leaf $E \in \Ft$ which intersects both $t_L$ and $t_F$. Notice that if $L, F \in \Ft$ are distinct
 non-separated leaves, then they cannot intersect a common foliation chart, so the distance between points in one leaf to the other leaf is bounded from below. 

We give some more definitions. We refer the reader to \cite[Section 3 and Appendix B]{BFFP2} for a broader introduction with similar notation. We can assume that the foliation is transversally oriented by going to a double cover and this makes no problem in our results since we are working in $\mt$. Given two leaves $L, F \in \Ft$,
the \emph{region between $L$ and $F$} is the intersection of the connected component of $\mt \setminus L$  containing $F$ and the connected component of $\mt\setminus F$ containing $L$. 

\begin{remark}\label{r.nestednonsep}
If $L$ and $F$ are non-separated and distinct,
then no transversal to $L$ can intersect $F$. Otherwise any leaf intersecting the transversal between $L, F$ would
separate $F$ from $L$.
\end{remark}

The following general result holds. 

\begin{prop}\label{p.sepbdcompact} 
Let $\cF$ be a Reebless foliation of a closed 3-manifold $M$. Assume that there are distinct leaves $L, F \in \Ft$ a finite Hausdorff distance apart, and which are non separated from each other (i.e. there is a sequence of leaves $L_n \in \Ft$ which converges to both in the leaf space).
Then, both $L$ and $F$ project into compact leaves in $M$. 
\end{prop}

\begin{proof}
Let $A$ and $B$ be the projections of $L$ and $F$ respectively to $M$. Assume that $A$ is not compact. Hence there is a sequence of points $p_i$ in $A$ such that $p_i \to p$ so that $p_i$ are not in the same plaques
of a local chart around $p$. Without loss of generality we can assume that 
the sequence $p_i$ is strictly monotone in the plaques of the chart.

One can lift the points $p_i$ to points $x_i \in L$ and consider $\gamma_i \in \pi_1(M)$ so that $\gamma_i x_i \to x_0$ a lift of $p$. 
Let $L_0$ be the leaf of $\Ft$ through $x_0$. The fact that the points $p_i$ converge to $p$ in different local leaves implies that $\gamma_i L$ are pairwise distinct leaves of $\Ft$ $-$ as transversals intersect a leaf only once in 
$\widetilde M$.  It is exactly this property that we will show
produces a contradiction.

Denote by $R>0$ a bound of the Hausdorff distance between $L$ and $F$. One can choose points $y_i \in F$ so that $d(y_i,x_i) < R+1$. Up to a subsequence, we can assume that $\gamma_i y_i \to y_0 \in F_0 \in \Ft$. Notice that $L_0 \neq F_0$ for otherwise one could fix a curve in $L_0$ from $x_0$ to $y_0$ and that would lift to nearby leaves, giving that $\gamma_i L$ and $\gamma_i F$ intersect the same transversal for large $i$ which is impossible since $L$ is non-separated from  $F$ and $L \not = F$
 (cf. Remark \ref{r.nestednonsep}). 

Now, pick transversals $t_{x_0}$ and $t_{y_0}$ to the leaves $L_0$ and $F_0$ through $x_0$ and $y_0$ respectively. For large $i$ it follows that the plaques through $\gamma_i x_i$ and $\gamma_i y_i$ intersect $t_{x_0}$ and $t_{y_0}$ respectively.  and so $t_{x_0}$ is a transversal to $\gamma_i L$ and $t_{y_0}$ a transversal to $\gamma_i F$. Since $\gamma_i L$ and $\gamma_i F$ are non-separated it follows that there are leaves intersecting both $t_{x_0}$ and $t_{y_0}$ which implies that $L_0$ and $F_0$ are non-separated
from each other. 

Assume first that $\gamma_i L$ does not belong to the region between $L_0$ and $F_0$. 

In this case $L_0$ separates $\gamma_i L$ from $\gamma_i F$ which is a contradiction since they are non-separated. In fact, if one considers a transversal $t$ to $\gamma_i L$ which is contained in the component of $\mt \setminus L_0$ not containing $F_0$ it follows that every leaf intersecting $t$ must remain in this component while $\gamma_i F$ must intersect a small transversal to $F_0$ so belong to a different connected component of $\mt \setminus L_0$ showing that $\gamma_i L$ and $\gamma_i F$ cannot be non-separated. The same works for $\gamma_i F$.

Suppose now that both 
$\gamma_i L$ and $\gamma_i F$ belong to the region between $L_0$ and $F_0$.
Recall that $(\gamma_i L)$ converges to $L_0$ and now $\gamma_i L$ is in the complementary component of $L_0$ containing $F_0$. In particular
the sequence $(\gamma_i L)$ also converges to $F_0$. 
Hence for $i$ big $\gamma_i L$ intersects $t_{F_0}$. 
Since for $i$ big the leaf $\gamma_i F$ also intersects
$t_{F_0}$ this would show that $\gamma_i F, \gamma_i L$ intersect
a common transversal contradicting the fact that $F, L$ not separated
from each other.

In other words, what these arguments really show is that the 
assumption that $\gamma_i L$ are all distinct leads to
a contradiction.

This finishes the proof of the proposition.
\end{proof}

The following result also holds in great generality. Notice that even if we assumed that $\cF$ is uniform, the result is not immediate since a priori we don't know if the region between two leaves has to be contained in a
bounded  neighborhood of one of the leaves. This is indeed what we show here for leaves which project into compact surfaces. 
Given a leaf $L$ of $\widetilde \cF$, let $\Gamma_L$ be the subgroup
of deck transformations fixing $L$, in other words, the stabilizer
of $L$ in $\pi_1(M)$. Notice that $\pi(L) = L/\Gamma_L$.

\begin{prop}\label{p.compactbd}
Let $\cF$ be a transversely oriented,
Reebless foliation of a closed 3-manifold. Let $L, F \in \Ft$ leaves at bounded Hausdorff distance whose projection to $M$ are compact surfaces. 
Let $\widetilde N$ be
the region between $L$ and $F$.
Then $\widetilde N$ projects to a compact $[0,1]$-bundle in $\widetilde M/\Gamma_L$.
\end{prop}

\begin{proof}
Notice that $\gamma F$ is at bounded Hausdorff distance from $L$ for every $\gamma \in \Gamma_L$ since deck transformations are isometries and $\gamma L= L$. As $F$ projects onto a compact surface, it follows that the orbit of $F$ by $\pi_1(M)$ is 
a closed subset of $\mt$. 

Let $R>0$ be the Hausdorff distance between $L$ and $F$ and consider a closed ball $B$ of radius $R+1$ centred at a point $x_0 \in L$. After covering $B$ with finitely many foliation charts, by compactness one sees that only finitely many translates of $F$ can intersect $B$. Since every translate of $F$ by some element of $\Gamma_L$ must intersect $B$, this implies that the action of $\Gamma_L$ in $F$ has finitely many translates of $F$. One deduces that the stabilizer of $F$ in $\Gamma_L$ is a finite index subgroup of $\Gamma_L$. 

The symmetric argument says that $\Gamma_F$ has a finite index subgroup fixing $L$. We deduce that $\Gamma = \Gamma_F \cap \Gamma_L$ is finite index in both $\Gamma_F$ and $\Gamma_L$. 

Consider the quotient $\mt/_{\Gamma}$ of $\mt$ by the group $\Gamma$. It follows that both $L$ and $F$ project to compact leaves in
$\mt/_{\Gamma}$.

The region $\widetilde N$ between $L, F$  projects to a compact 3-manifold with boundary $N_{\Gamma}$ in $\mt/_{\Gamma}$,
whose boundaries are the quotients $A$ and $B$ of $L$ and $F$. 
Since $\cF$ is Reebless and transversely oriented then leaves
of $\cF$ are $\pi_1$ injective in $\pi_1(M)$, so $\pi_1(A), \pi_1(B)$
inject into $\pi_1(N_{\Gamma})$. Clearly $\pi_1(A), \pi_1(B)$ also
surject into $\pi_1(N_{\Gamma})$.
In addition $N_{\Gamma}$ is irreducible.

It follows (see \cite[Theorem 10.2]{Hempel}) that $N_{\Gamma}$ is homeomorphic to $A \times [0,1]$ and $A \times \{1\}$ corresponds to $B$.  
Projecting to $\widetilde M/\Gamma_L$ one gets that $\widetilde N$
also projects to an $[0,1]$-bundle with a boundary a leaf homeomorphic
to $C = L/_{\Gamma_L}$. This uses that $\cF$ is
transversely oriented.
\end{proof}

We need one additional result.

\begin{prop}\label{p.t2i}
Suppose that $\cF$ is a Reebless foliation in $N = \TT^2 \times [0,1]$, so that
each boundary component is a leaf of $\cF$. Suppose that in 
$\widetilde N$, the boundary leaves $E = \widetilde \TT^2 \times \{ 0 \}$
and $G = \widetilde \TT^2 \times \{ 1 \}$ are not separated from each other
in the leaf space of $\widetilde \cF$. Then $\cF$ is not 
a uniform foliation.
\end{prop}

\begin{proof}
Given a leaf $L$ in the interior of $\widetilde N$, we will 
show it cannot be a finite Hausdorff distance from either 
one of the boundary
leaves. 
Since $N$ is a product there is $b_0 > 0$ so that
$\widetilde N$ is  contained in the neighborhood of size
$b_0$ of $E$, and likewise for $G$. 
We will show that $E$ cannot be in a bounded neighborhood
of any such $L$ as above.

Lifting to a double cover if necessary
we can assume that $\cF$ is transversely
orientable.

Since $\cF$ is Reebless the fundamental group of leaves injects
in $\pi_1(N)$, so the leaves are either planes, annuli
or tori. 
If there is a compact leaf in the interior of $N$,
then its fundamental group injects in $\ZZ^2 = \pi_1(N)$, so it is a torus, and hence it 
is isotopic to $\TT^2 \times \{ 0 \}$. It lifts to
a leaf $Z$ in $\widetilde N$ which separates $E$ from $G$,
contradiction. 
So the leaves in the interior of $N$ are only
planes and annuli.

Let $A = \TT^2 \times \{ 0 \}$,
$B = \TT^2 \times \{ 1 \}$.
We look at the holonomy of $\cF$ along a boundary leaf,
say $A$.
We want to find an element of $\pi_1(A)$ with contracting holonomy.
Fix $x$ a basepoint in $A$, let $\tau$ be a small transversal
to $\cF$ at $x$.
Let $\alpha$ represent a simple closed curve in $A$ not
null homotopic. If either $\alpha$ or $\alpha^{-1}$ has 
contracting holonomy, that is the element we want. 
Otherwise there are $p_i$ in $\tau$ converging to $x$ so
that $\alpha$ holonomy fixes $p_i$.  Fix $i$, let $C$ be the leaf
through $p_i$. Then $C$ is an annulus. Let now $\beta$ another
simple closed curve which generates $\pi_1(\TT^2)$
together with $\alpha$. If holonomy of $\beta$ fixes $p_i$
also then $C$ is in fact a compact leaf, but in the
interior of $N$, which we showed it is not possible.
So replacing $\beta$ by its inverse, the holonomy image
of $p_i$ under $\beta$ is closer to $x$. If the iterates
converge to $x$, then $\beta$ is the desired element.
Otherwise the iterates converge to $y$ not $x$, and the
leaf through $y$ is compact, again a contradiction.

Let then $\alpha$ be a simple closed curve in $A$ with
contracting holonomy. 
We think of $\alpha$ also as a deck transformation. 
Then $\alpha$ fixes $E$.

Fix a point $y$ in $E$ and a 
transversal $\tau$. Since holonomy of the foliation $\cF$
is contracting in the $\alpha$ direction this means
that $\alpha^{-1}(L)$ intersects $\tau$ and in a point closer
to $E$. The contracting holonomy means that the sequence
$(\alpha^n(L))$ converges to $E$ as $n \rightarrow -\infty$.
In fact this is an if and only if property: if there is 
$L$ intersecting $\tau$ so that $(\alpha^n(L))$ converges to
$E$ as $n \rightarrow -\infty$, then $\alpha$ has
contracting holonomy.

But $\alpha$ also preserves $G$. Since $E, G$ non separated
from each other, and $(\alpha^n(L))$ converges to $E$, it follows
that $(\alpha^n(L))$ also converges to $G$ when $n$ converges
to minus infinity. By the if and only if characterization above,
this implies the following:
If $\beta$ is a simple closed curve in $B$ freely homotopic to
$\alpha$ then the holonomy of $\cF$ along $\beta$ is contracting
as well.

\vskip .1in
We proceed with the proof of the proposition.
We consider a model of $N$ as $\TT^2 \times [0,1]$ so that $\widetilde N$
is homeomorphic to $\RR^2 \times [0,1]$ with coordinates
$(a,b,c)$ and any deck transformation acts as $(\theta(a,b),c)$,
where $\theta$ is a translation of $\RR^2$.
In that way we can choose coordinates so 
that $\alpha(a,b,c) = ((a,b)+(1,0), c)$.

Suppose now that $E$ is in a neighborhood of size $a_0$ of $L$.
For any  $n$ there is a point $p_n$ in $L$ which is $< a_0$ distant
from $(-n(1,0),0)$. 

\begin{claim} Given $\epsilon > 0$, there is $a_1 > 0$ so that
any point in a leaf $U$ of $\widetilde \cF$, it is less than
$a_1$ along $U$ from a point $\epsilon$ distant  from $E$ or $G$.
\end{claim}

\begin{proof} Suppose not. Project to $N$, we get bigger and bigger sets in
leaves which avoid an $\epsilon$ neighborhood of the boundary.
Taking a limit we find a leaf $V$ of $\cF$ avoiding an $\epsilon$
neighborhood of the boundary. The closure of $V$ is a lamination
in $N$ disjoint from the boundary. It is an essential
lamination $W$. Double $N$ to get a Seifert fibered space,
$W$ is still an essential lamination. By Brittenham's
result \cite{Brit}, $W$ has a sublamination that is either vertical
or horizontal in the double of $\TT^2 \times I$.
If $W$ is vertical it would have to intersect
a boundary component of $N$. This is a 
 horizontal $\TT^2$ in the double manifold.
This is a contradiction.
Suppose that $W$ is horizontal. It is also contained in $\TT^2 \times I$,
hence a ``topmost" leaf would have to be compact, hence a torus.
This is contained in the interior of $N$, again a contradiction.
This proves the claim.
\end{proof}

We fix $\epsilon > 0$ so that the foliation $\cF$ restricted
to the $\epsilon$ neighborhood of the boundary of $N$ is entirely
described by the holonomy maps. Let $a_1 > 0$ given by the claim.
So given $n$, there is $q_n$ in $L$, which is less than $a_1$
along $L$ from $p_n$ and $q_n$ is $\epsilon$ away from the 
boundary. Hence $q_n = ((-n,0)+ v_n, t_n)$ where $v_n$ is
bounded under $n$ and $|t_n| < \eps$ or 
$|t_n| > 1 - \eps$. Up to subsequence we assume that all 
$v_{n_i}$ are very close to $v_0$ (projection to $N$ all in a fixed
foliated chart).

Now apply the holonomy of $\alpha^n$ to $q_n$. Since $q_n$ is
$\epsilon$ close to the boundary and the holonomy of $\alpha$ 
is contracting in the neighborhood of size $\epsilon$ of
both $A$ and $B$ it follows that the holonomy image of $q_{n_i}$
is $(v_{n_i}, t_{n_i})$ where $t_{n_i}$ is either arbitrarily
close to $0$ or to $1$. None is either $0$ or $1$ as $L$ is in
the interior of $N$. They are all points in $L$, and this 
contradicts
that $L$ cannot intersect a transversal more than once.

This contradiction proves that the assumption that $E$ is at a
bounded distance from $L$ is impossible.
Hence the foliation $\cF$ is not uniform.
\end{proof}

\subsection{Proof of Theorem \ref{teo-Rcovered}}

Now we are ready to prove Theorem \ref{teo-Rcovered}. Let $\cF$ be a uniform Reebless foliation on $M$. We want to show that $\cF$ is $\R$-covered, so we assume by contradiction that there are leaves $L$ and $F$ of $\Ft$ which are non-separated in the leaf space $\cL_\cF = \mt/_{\Ft}$ of $\cF$. 
Up to a double cover we may assume that $\cF$ is transversely
oriented.

Proposition \ref{p.sepbdcompact} implies that both $L$ and $F$ project to compact surfaces in $M$. Let $\Gamma_L$ be the stabilizer of $L$ in
$\pi_1(M)$.  Proposition \ref{p.compactbd} shows that the region $\tild{N}$ between $L$ and $F$ projects to a compact $[0,1]$-bundle $W$ in $\widetilde M/\Gamma_L$, with one boundary $L/\Gamma_L$.

Suppose that there is a deck translate $\beta(L)$ of $L$ or $F$ inside $\tild{N}$.
It projects to a surface in $\widetilde M/\Gamma_L$ 
contained in the $[0,1]$-bundle $W$. 
Since $\pi(L)$ is compact in $M$, then $H = \beta(L)/_{\Gamma_L}$ is
also compact.
Since $H$ is $\pi_1$-injective
in $W$ it follows that $H$ is isotopic in $W$ to a boundary 
component. Lifting to $\widetilde M$ this implies that $\beta(L)$
separates $F$ from $L$, contradicting that they are non separated.

Let $A = \pi(L)$. 
Suppose that there is a closed transversal to
$\cF$ through $A$. 
Lift to $\widetilde M$, 
with the transversal intersecting $L$ and entering $\tild{N}$.
It cannot exit $\tild{N}$ as $F,L$ do not intersect a
closed transversal. Hence this produces a deck translate of $L$
inside $\tild{N}$ which we just proved cannot happen.
Hence there are no closed transversal through either $A$
or $B = \pi(F)$. 

On the other hand suppose there are $E_i$ converging to $F \cup L$
so that $\pi(E_i)$ is compact. For $i$ big enough $\pi(E_i)$ is 
isotopic to $A$, and hence $E_i$ separates $F$ from $L$, contradiction.
Hence $\pi(E_i)$ is non compact and there are transversals through
$\pi(E_i)$ for $i$ big enough. It follows that the region between
$A$ and $B$ is a dead end component, see \cite[Definition 4.27]{Calegari}. By  \cite[ Lemma 4.28]{Calegari},
$A, B$ are two sided tori or Klein bottles. Lifting to 
a double cover we can assume that both $A, B$ are tori.

It can be that $A = B$, but in any case $\tild{N}$ 
projects in $\widetilde M/\Gamma_L$ to a compact submanifold homeomorphic to $\TT^2 \times [0,1]$.

We can now apply Proposition \ref{p.t2i}.
Let $G$ be a leaf in $\tild{N}$. 
By Proposition \ref{p.t2i} it follows that $L$ is not
a bounded distance from $G$ in $\tild{N}$.
Suppose that this does not happen in $\widetilde M$.
Then there are points $p_i$ in $L$ which are $> i$ distant
from $G$ along path distance in $\tild{N}$,
 but a bounded distance in $\widetilde M$
from $q'_i$ in $G$. Notice that $q'_i$ is a bounded distance 
in $\tild{N}$ from $q_i$ in $L$ $-$ just follow along the
lift of the $I$-bundle structure to $\tild{N}$. If
one uses the parametrization $(a,b,c)$ as in Proposition
\ref{p.t2i} one can assume up to moving them boundedly
in $L$, that $p_i,q_i$ have all coordinates integers and the last
coordinate $0$. Consider a generating set of $\pi_1(M)$ which
includes $2$ generators of the torus $A$. Then $p_i, q_i$
are vertices of the Cayley graph. Modulo deck transformations
sending $p_i$ back to a base point, it follows that $q_i$
is a  bounded neighborhood of the origin. So only finitely
many elements of $\pi_1(M)$ are allowed. It follows that
$q_i$ is a bounded distance from $p_i$ along $L$.
This is a contradiction.

This completes the proof of Theorem \ref{teo-Rcovered}.

\section{Universal circles and JSJ trees}\label{s.UnivJSJ}
In this section we will show that for $\R$-covered foliations (uniform or not)
one can recover the universal circle from the JSJ decomposition of the manifold (cf. Proposition \ref{prop-JSJvsUniversalCircle}), if the manifold has a non trivial JSJ decomposition. This will allow us to prove Proposition \ref{p.movingpoints} that we will need in the proof of Theorem \ref{teo-Minimal}.  Proposition \ref{p.movingpoints} states that the action of the fundamental group on the universal circle does not have fixed points which is certainly a fact that needs to be established if one desires to obtain minimality of the action. 

Consider an $\R$-covered foliation $\cF$ by leaves with 
curvature uniformly close to $-1$ on a closed 3-manifold $M$, so that $M$ has non trivial JSJ decomposition. In particular the leaves
are Gromov hyperbolic. 
If $\cF$ is not taut, then there are dead end components,
see \cite[Definition 4.27]{Calegari}. In particular there
are either tori or Klein bottle leaves. This is disallowed
by $\cF$ having Gromov hyperbolic leaves.
Hence $\cF$ is taut.

We will consider that $M$ is orientable and $\cF$ transversely
orientable. The only difference in the non-orientable case is that in the 
JSJ decomposition we also have to consider Klein bottles.
These Klein bottles lift to embedded tori in some cover
of $M$. Then all the results follow with the same proofs.

\subsection{The trace of JSJ tori in the universal circle}

Let $M_1, \ldots M_k$ be the pieces of its JSJ decomposition. 
Let $T$ be a torus of the JSJ decomposition. In this section we show Proposition \ref{prop-tracetori} which states that one can associate to each lift of a torus of the JSJ decomposition some points in the universal circle. 

We first need the following lemma that puts (after isotopy) the JSJ tori in general position.

\begin{lemma}{\label{lem-goodposition}}
Any lift $\widetilde T$ to $\widetilde M$ intersects
every leaf of $\widetilde \cF$.
In addition one can isotope $T$ so that $\widetilde T$ intersects
every leaf of $\widetilde \cF$ in a single component,
and so that the foliation induced by $\cF$ in $T$ has
no Reeb components.
\end{lemma}

\begin{proof}
Let $G = \Z^2$ be the isotropy group of $\widetilde T$. 
The set of $\widetilde \cF$ leaves intersected by
$\widetilde T$ is connected. If this set is not the whole
leaf space, it is a non trivial interval in the leaf space.
Let $F$ be an endpoint. Since the leaf space is homeomorphic
to $\R$, it follows that $G$ preserves $F$. So $\pi_1(\pi(F))$
has a $\Z^2$ subgroup and the projection 
$\pi(F)$ is therefore a torus or Klein bottle.
This contradicts that the leaves of $\cF$ are Gromov hyperbolic.

Since $\cF$ is taut, by Theorem \ref{teo-transversetori} we
can isotope $T$ to be either a leaf of $\cF$ or transverse
to $\cF$. The first option is disallowed because of Gromov
hyperbolic leaves.
Hence assume that $T$ is transverse to $\cF$, let $\cG$ be
the induced foliation in $T$.

\begin{claim}
It is possible to isotope $T$ so that $\cG$ has no
{\em Reeb annuli}.
\end{claim}

\begin{proof} 
A Reeb annulus 
is a foliation of the annulus  so that boundaries 
are leaves, all other leaves spiral toward the boundary
leaves, and there is no transversal arc intersecting both
boundary leaves.
Suppose that $\cG$ has a Reeb annulus $A$.
The two boundary leaves of $A$ lift to curves
in $\widetilde M$, contained in leaves of $\widetilde \cF$
which are non separated from each other. This is because
of the Reeb annulus, so in $\widetilde A$ the boundary
infinite lines are non separated from each other.
Since the foliation is $\R$-covered, the two leaves
of $\widetilde \cF$ containing these infinite lines 
$\alpha, \beta$ are
the same leaf $L$.
Since $\pi(\alpha), \pi(\beta)$ are freely homotopic
in $T$, then $\alpha, \beta$ are a bounded distance from
each other in $\widetilde M$. 
We now use a fact of $\R$-covered foliations: for any 
$a_0 > 0$, there is $a_1 > 0$, so that if two points $x, y$
in a leaf $F$ of $\widetilde  \cF$ are less that $a_0$ in
$\widetilde M$, then they are less than $a_1$ in $L$ (see
\cite[Proposition 2.1]{FenleyQI}).

This holds only for $\R$-covered foliations.
Hence $\alpha, \beta$ are a bounded distance from each
other in $L$. It now follows that $\pi(\alpha), \pi(\beta)$
are isotopic closed curves in $\pi(L)$ and bound 
an annulus $B$ in $\pi(L)$. The interior of $B$ cannot intersect
$A$, because any interior leaf of $\cG$ in $A$ limits to
the boundary of $A$, and $A, B$ are transverse to each other.
Hence $A \cup B$ is
a torus. This torus is not $\pi_1$ injective because
one can produce an essential arc across $A$ together with
one across $B$ to yield a closed curve which is null homotopic.
One can easily see this as $\widetilde B$ is contained
in the fixed leaf $L$, and $\widetilde A$ has both
boundaries in $L$. Hence $A \cup B$ is compressible and there
is a compressing disk $D$ intersecting $A \cup B$ only in the boundary.
Cutting $A \cup B$ along $D$, produces a sphere. Since
$M$ is irreducible, this sphere bounds a ball. Gluing
back together one sees that $A \cup B$ bounds a solid torus.

What we proved is that $B$ is isotopic to $A$ in $M$. So 
then one can isotope $A$ across the solid torus to the
other side of $B$ and eliminate this Reeb annulus in $\cG$.
Doing this finitely many times eliminates all Reeb annuli in
$\cG$. 
This proves the claim.
See also \cite[Theorem 5.3.13]{CalegariPA} for a similar statement. 
\end{proof}

\vskip .1in
Since there are no Reeb annuli in $\cG$, it follows that 
$\cF$ intersects $T$ in a foliation uniformly equivalent\footnote{By this we mean that in the universal cover each leaf is bounded Hausdorff distance from a leaf of the linear foliation. See \cite[Definition 2.1]{Thurston} for a general definition of being uniformly equivalent.} to a linear foliation of the two dimensional torus. 
In particular any two leaves of $\widetilde \cG$ are connected by a 
tranversal to $\widetilde \cG$, hence a transversal to
$\widetilde \cF$ as well. It follows that any leaf $F$ of $\widetilde
\cF$ intersects $\widetilde T$ in a single component.

This finishes the proof of the lemma.
\end{proof}

\begin{remark}\label{rem-QItori}
The reason we choose the definition of non-trivial JSJ decomposition is to exclude $Sol$ and $Nil$ geometries for which some of the arguments do not
work.  These cases are not problematic to us and can be dealt with separately, and in a different way. A good thing about manifolds with non-trivial JSJ decomposition under our definition is that the tori of the decompositons are quasi-isometrically embedded: the map between the universal covers is a quasi-isometric embedding.
This follows from \cite[Theorem 1.1]{KL} 
(see also \cite[Section 3.1]{Ng}).  
In particular when 
 lifted to $\mt$, every quasigeodesic 
in the lift of the 
torus lifts maps to a quasigeodesic in $\mt$. 
\end{remark}

Let $T$ be a torus of the JSJ decomposition, put in good
position as in Lemma \ref{lem-goodposition}.
Let $\cG$ be the induced foliation by $\cF$ in $T$.
Given $L$ leaf of $\widetilde \cF$, and a lift $\widetilde T$ of $T$, then by Remark \ref{rem-QItori},
the curve $L \cap \widetilde
T$ is a quasigeodesic of $\mt$. It is also a leaf of
$\widetilde \cG$. 
Since it is a quasigeodesic in $\mt$, then it is necessarily
also a quasigeodesic in $L$, with ideal points $a_L(\widetilde T), b_L(\widetilde T)$
in $S^1(L)$. Orient the foliation $\cG$ so that
$b_L(\wt{T})$ corresponds to the forward direction in $\cG$.
Varying the leaf, produces corresponding ideal points
$a_F(\wt{T}), b_F(\wt{T})$ in $S^1(F)$ for any
$F$ leaf of $\widetilde \cF$. 

\begin{proposition}\label{prop-tracetori}
The collection $\{ b_F(\wt{T}) \}$ as $F$ varies over
leaves of $\widetilde \cF$ is a leaf of the vertical
foliation in the cylinder at infinity $\cA$. Equivalently, the point $\{b_F(\wt{T})\}$ is well defined in $\suniv$ and independent of the leaf $F$. 
\end{proposition}

\begin{proof} We will fix a lift $\wt{T}$ of some torus $T$ of the JSJ decomposition. So, we will not include the reference to $\wt{T}$ in the notation. 

Suppose first that $\cF$ is uniform.
Let 
$\alpha_L$ be the intersection of $L$ and $\widetilde T$, that
is a leaf of $\widetilde \cG$.
For any $L, F$ leaves of $\widetilde \cF$, 
the curves $\alpha_L, \alpha_F$ are a bounded distance
from each other in $\widetilde T$ $-$ since there are
no Reeb annuli in $\cG$. It follows that $\alpha_L, \alpha_F$ are
a bounded distance from each other in $\mt$.
By the remark above, $\alpha_L$ is a quasigeodesic in $L$, hence, the ray
$\beta_L$  defining $b_L$ is a bounded distance
in $L$ from a geodesic ray in $L$. Since $\cF$ is uniform,
this ray in $L$ is a bounded distance from a geodesic ray
in $F$ defining $\tau_{L,F}(b_L)$. But $\beta_L$ is a bounded
distance from a corresponding ray
$\beta_F$ of $\alpha_F$ (same direction given by the foliation $\cG$).
This is bounded distance in $\mt$. Hence $\beta_F$ is a bounded
distance in $\mt$ from the geodesic ray defining $\tau_{L,F}(b_L)$.
Since $\cF$ is $\R$-covered, this again implies that 
$\beta_F$ is a bounded distance from this geodesic ray in $F$.
In particular the ideal point of $\beta_F$ is $\tau_{L,F}(b_L)$.
But by definition the ideal point of $\beta_F$ is $b_F$.
Hence $b_F = \tau_{L,F}(b_L)$. This proves the
proposition in this case.

\vskip .1in
Suppose now that $\cF$ is not uniform.
By the description in \S \ref{sss.notuniform}
we can assume that $\cF$ is
minimal. Hence for any $L, F$ in $\widetilde \cF$ there is 
a dense set of directions in $S^1(L)$ which are asymptotic 
to $F$.

Fix a transversal $\tau$ to $\cG$  in $T$. Lift this to a transversal
$\tilde \tau$ in $\widetilde T$. For any $L$ intersecting
$\tilde T$, let $x_L = \tilde \tau \cap L$. Let 
$r_L$ be the geodesic ray in $L$ starting at $x_L$ and with
ideal point $b_L$. As $L$ varies the corresponding rays $\beta_L$ in
$\widetilde \cG$ are boundedly close to each other in
$\widetilde T$ and hence in $\widetilde M$. 
Hence the same happens
for the geodesic rays $r_L$ as $L$ varies. It follows
that the ideal points of $\beta_L$ vary continuously
with $L$. Hence the functions $a_L, b_L$ from the leaf
space into $\cA$ are continuous.

Suppose that for some $L, F$, then $\tau_{L,F}(b_L) \not = b_F$.
Since the set of contracting directions between $L$ and $F$
is dense in $S^1(L)$ and $b_E$ varies continuously with $E$,
it follows that there is some $E$ between $L, F$ so that
$b_E$ corresponds to a direction in $E$ which is contracting
with both $L$ and $F$. 
Hence the ray $\beta_E$ in $E \cap \widetilde T$ is asymptotic
to a curve in $L$. This implies that in $\widetilde T$, the curve
$\beta_E$ is asymptotic to a curve in $\widetilde T \cap L$.
But this can only be $\beta_L$ $-$ as $\widetilde T \cap L$
is a single curve and has a ray $\beta_L$ corresponding
to that direction. In particular this implies that 
$b_E = \tau_{L,E}(b_L)$. The same holds for the pair $E, F$.
By the composition property of the maps $\tau_{L,F}$, it now follows
that $\tau_{L,F}(b_L) = b_F$. 

This finishes the proof of the proposition.
\end{proof}

\subsection{JSJ universal circles}

Our setup has an $\R$-covered  foliation $\cF$ by leaves with curvature
very close to $-1$ in $M$ with non trivial JSJ decomposition.
If $T$ is a torus in the JSJ decomposition we use
Lemma \ref{lem-goodposition} and isotope $T$ 
to be transverse to $\cF$ and so that the induced foliation
in $T$ does not have any Reeb annuli. 

Recall that in Proposition \ref{p.JSJtree} we introduced
the JSJ tree $\cT$ of $M$.
Let $T_1, ..., T_k$ be the tori in the JSJ decomposition.
The fundamental group of $M$ naturally acts on the tree $\cT$.
The tree $\cT$ is infinite and in general not locally compact:
there are infinitely many edges adjoining any given vertex.
We observe that if $M$ has a trivial JSJ decomposition,
that is, $M$ is either Seifert or atoroidal, then the object
constructed above would be a single point. 
We now consider the case that $M$ has an $\R$-covered foliation.

Let $W = \pi^{-1}(T_1 \cup \ldots \cup T_k)$. In other
words a component of $W$ is an arbitrary lift 
$\widetilde T$ of one of the JSJ tori.

\begin{lemma}{}\label{l.JSJinleaves}
Suppose that $M$ has a non-trivial JSJ decomposition and $\cF$ is an $\R$-covered foliation by leaves with curvature very close to $-1$.

Then the JSJ tree $\cT$ has an
embedding into the plane well defined up to isotopy. This
determines a well defined circular ordering 
on the set of ends of $\cT$.
A deck transformation either preserves the circular ordering,
or reverses the circular ordering on the set of ends.
\end{lemma}

\begin{proof}{}
The curvature condition implies that $\cF$ is Reebless.

Hence the leaves
of $\widetilde \cF$ are properly embedded planes in $\mt$.

First fix a leaf $F$ of $\widetilde \cF$.
Lemma \ref{lem-goodposition} shows that any lift $\widetilde T$ of
a JSJ torus intersects $F$ in a single component. This component
is a quasigeodesic in $F$.
For each vertex $y$ of $\cT$, associated to a component $V$ of
$\mt - W$, it has at least two edges adjoining
it, let $\widetilde T$ be one of them. Since $\widetilde T$ intersects $F$
transversely, then $V$ also intersects $F$. In addition since
any lift $\widetilde T'$ of a JSJ torus separates $\mt$, and
each such lift intersects $F$ in a single component, it also
follows that $V$ also intersects $F$ in a single component.
Choose a point $p_V$ in $V \cap F$ representing the vertex 
$y$ of $\cT$. It $\widetilde T$ is an edge of $\cT$ adjoining
components $V$, $Z$ of $\mt - W$, choose an embedded arc in $F$
connecting $p_V$ to $p_Z$, and intersecting $\widetilde T$
in a single point. This represents an embedding of
the edge $\widetilde T$ of $\cT$ into $F$. In this way
we construct an embedding of $\cT$ into $F$.
The choices of the points $p_V$ are well defined up to
isotopy in $V \cap F$. The choices of the embedded arcs are
also well defined up to isotopy.
Therefore the embedding of $\cT$ into $F$ is well defined up
to isotopy.
Fix one such embedding and call $\cT_F$ the image tree in $F$.

Now if $L$ is another leaf of $\widetilde \cF$, then the
same reasoning applies. Notice that if $V$, $Z$ components
of $\mt - W$ define and edge $\widetilde T$, then 
$V \cap L, Z \cap L$ are adjoining in $L$ along $\widetilde T
\cap L$ just as in $F$. In addition the circular ordering
around a vertex is also the same whether considering it
wrt to $F$ or to $L$. It follows that the embeddings of $\cT$
in $F$ and $L$ are isomorphic, preserving the circular ordering
at the corresponding vertices.

It follows that the embedding in the plane is well defined
up to isotopy.
This induces a circular ordering in the set of ends of
$\cT$. 

If $\gamma$ is a deck transformation, and $F$ a leaf of 
$\widetilde \cF$, then $\gamma$ also induces a homeomorphism
of the embedding of $\cT$ in $F$:
given $V$ components of $\mt - W$, then $\gamma(V)$ also
intersects $F$ in a single component, and likewise for
$\widetilde T$ component of $W$.
This produces the required homeomorphism of the 
tree $\cT$. In addition this homeomorphism is induced
by a homeomorphism between $F$ and $\gamma(F)$, which
can be either orientation preserving or reversing.
It follows that this homeomorphism
either preserves the circular ordering of the ends of $\cT$
or reverses it.
\end{proof}

\begin{remark}
We emphasize some facts proved in this lemma: if $V$ is a component
of $\mt - W$, and $F$ is a leaf of $\widetilde \cF$, then 
$V$ intersects $F$ and in a single component (cf. Lemma \ref{lem-goodposition}). Similarly if 
$\widetilde T$ is a component of $W$ then $\widetilde T$ intersects
$F$ in a single component. Therefore the trees $\cT$ and $\cT_F$ 
are canonically isomorphic. In particular if $F, L$ are
leaves of $\widetilde \cF$, then $\cT_F, \cT_L$ are canonically
isomorphic, with the circular order of the edges at any vertex
preserved by the isomorphism (see also Proposition \ref{prop-tracetori}). 
\end{remark}

We produced a set with a circular order and a group action
so that each group element either preserves the circular
order or reverses it.
Given these properties, a circle with an induced action 
can be created.
This procedure from set with circular order and group action
to action on a circle was developed 
by Calegari and Dunfield in
\cite{CD}. We refer to  \cite[Theorem 3.2]{CD} for specific
details. Here we will only briefly describe the construction of the
circle with the induced action.

Since the set of ends is cyclically ordered there is an
embedding of the set of ends into a circle preserving
the circular order. First take the closure of the image of the set of ends.
If the tree were locally finite (finitely many edges at any
vertex), then the set of ends would be order complete,
and the image is a closed subset of the circle..
The fundamental group still acts on the closure. 
There may be gaps in the image. 
Now collapse every closure
of a complementary interval (that is a gap) to a point, producing a 
circle $\tj$, called the 
{\em {JSJ universal circle}} of $\cF$. 
Deck transformations either preserve or reverse the circular
ordering so induce homeomorphisms of the circle that either
preverse or reverse orientation. 

\begin{remark}
The JSJ universal circle depends on the foliation
$\cF$: given a different $\R$-covered foliation $\cF_1$, 
it may induce a different circular ordering of the 
edges at a given vertex of the tree $\cT$.
This will produce a different circular order on the set of
ends of $\cT$ and hence a different JSJ universal circle.
The tree $\cT$ is the same and so are its ends. But the
the set of edges around a vertex in $\cT$ does not come
with a natural circular order. This is the information that
the $\R$-covered foliation is providing, because
it gives an embedding of the tree into the plane.
Different
$\R$-covered foliations may give different such 
circular orders.
\end{remark}

Let $T$ be a $\pi_1$-injective torus in $M$, put in 
good position as in Lemma \ref{lem-goodposition}.
Given $F$ leaf of $\widetilde \cF$, we define 
the 
lamination $\cG_F$ whose leaves are the intersections
of lifts $\widetilde T$ of $T$ with $F$.
In fact $\cG_F$ also depends on $T$, but for notational
simplicity we omit this dependence.

\begin{lemma}{}\label{denseidealpoints}
For each  $\pi_1$-injective torus $T$ of $M$ and for each
$F$ leaf of $\widetilde \cF$, then the set of ideal 
points of leaves of $\cG_F$ is dense in ${S}^1(F)$.
In addition for any non degenerate interval $J$ of
$S^1(F)$ there are leaves of $\cG_F$ with both
ideal points in $J$.
\end{lemma}

\begin{proof}{}
Suppose the first property is not true, let $T$ be a $\pi_1$-injective torus and $F$ a leaf of $\widetilde{\cF}$ so that the set of ideal points of leaves of $\cG_F$ is not dense in $S^1(F)$.

Then there is a non trivial interval $I$
in $S^1(F)$ which is disjoint from the ideal points of 
of $\cG_F$. Since the curves in $\cG_F$ are uniform
quasigeodesics in $F$ they are a uniform bounded distance
from geodesics in $F$. Hence up to considering a subinterval,
it follows that $I$ bounds a half plane $P$ in $F$ which
is disjoint from $\cG_F$. Therefore there are disks $D_i$
with radius converging to infinity disjoint from $\cG_F$.
Up to taking subsequences and deck transformations $g_i$,
then $g_i(D_i)$ converges to a full leaf $L$ which
is disjoint from $\cG_L$. But this is impossible since
any lift $\widetilde T$ of $T$ intersects every leaf of 
$\widetilde \cF$. This proves the first property of the lemma.

Now suppose that $J$ is a non degenerate interval so that 
no leaf of $\cG_F$ has both ideal points in $J$.
Let $x$ be an interior point of $J$. Let $x_i$ a sequence
of distinct points in $J$ converging monotonically 
to $x$. There are leaves
$c_i$ of $\cG_F$ with an ideal point arbitrarily
close to $x_i$. Since the $x_i$ are distinct in $J$ we can choose
the $c_i$ to be distinct as well. The other endpoints of $c_i$
are not $J$, hence at least $a_1 > 0$ from the first endpoint
of $c_i$ which is arbitrarily close to $x$. Since the 
$c_i$ are uniform quasigeodesics, then up to subsequence we
may assume that $c_i$ converges to a quasigeodesic $c$.
But then different $c_i, c_j$ have points that
are arbitrarily close to  each other. This is a contradiction:
different tori in the JSJ decomposition are compact and
disjoint. This implies that there is a constant $a_2 > 0$,
so that 
if $C, C'$ are different lifts of JSJ tori, then points
$p \in C, p' \in C'$ satisfy that distance from 
$p$ to $p'$ is at least $a_2$.

This finishes the proof of the lemma.
\end{proof}

We can now prove the following proposition that gives a different way to think about the universal circle of a foliation in terms of the JSJ universal circle. 

\begin{proposition}\label{prop-JSJvsUniversalCircle} Suppose that $M$ has a non-trivial JSJ
decomposition and $\cF$ is an $\R$-covered foliation with
Gromov hyperbolic leaves. Then there is a canonical homeomorphism
between the universal circle $\suniv$ of $\cF$ and the JSJ universal
circle $\tj$ of $\cF$.
This homeomorphisms is equivariant under deck transformations.
\end{proposition}

\begin{proof}{}
For simplicity fix a leaf $F$ of $\widetilde \cF$.
The universal circle of $\cF$ is canonically identified with
$S^1(F)$. The JSJ universal circle can be obtained from the
intersections with $F$.
What we will prove is that considering $F$, both of these
are canonically homeomorphic.

Let $\cT_F$ be the embedded tree in $F$ which is the homeomorphic
image of $\cT$. Fix a basepoint $p$ in $\cT_F$.
Let $\cB$ be the set of ends of $\cT_F$. Since $\cT_F$ is
a tree it is easy to see that each end is uniquely associated 
to an embedded ray in $\cT_F$ starting at $p$.
Let $e$ be an end in $\cB$ associated to a ray $\alpha$
in $\cT_F$, which is also an embedded ray in $F$.
Then $\alpha$ keeps intersecting lifts $C_i$ of
one of the JSJ tori, let $c_i = C_i \cap F$.
Recall that $c_i$ is a quasigeodesic with uniform constants,
so globally $a_0$ distant from a geodesic in $F$.
Any two lifts $C, C'$ of JSJ tori have a minimum separation
between them. Hence the corresponding points $C \cap F, C' \cap F$
also have a minimum separation between them. 
Therefore the geodesics associated to $c_i$ also escape
in $F$ and they define a unique ideal point in $S^1(F)$
which we call $f(e)$. 
This defines a map $f$ from the set of ends $\cB$ 
to $S^1(F)$. 

Given appropriate orientations on $S^1(F)$ and the circular
order on the set of ends of $\cT_F$, it follows that
the map $f$ preserves this circular order.
In particular as one goes around once in the circular order
of the ends of $\cT_F$, then one also goes around once
in $S^1(F)$.
By Lemma \ref{denseidealpoints}, for each non degenerate
interval $J$ in $S^1(F)$
there is a leaf $c$ of $\cL_F$ with both ideal points  in $J$.
Hence any end $e$ of $\cT_F$ which is associated with a path in
the tree $\cT_F$ which
crosses $c$ will have $f(e)$ in $J$. It follows that
the image of $f$ is dense in $S^1(F)$.

Recall the construction of the JSJ universal
circle $S^1_{JSJ}$ of $\cF$: we map the set of ends $\cB$ to a circle $S^1$ preserving
the circular order, take the closure and then collapse the
gaps. 

By the first step we can think of $\cB$ as a subset of $S^1$.
 Let $H$ be the closure in $S^1$ of the image
of $f$.
Since $f$ preserves circular order it induces a map $f_1$ from
$H$ into $S^1(F)$. This map is
weakly monotone.
Since the image of $\cB$ under $f$ is dense in $S^1(F)$
it follows that given the endpoints of a gap of $H$ they
have the same image in $S^1(F)$ under $f_1$.
This implies that $f_1$ induces a map $f_*$ from the JSJ
universal circle $\tj$ of $\cF$ to $S^1(F)$. 

Finally by the same reasoning 
if two points have the same image under $f_1$
then they have to be boundary points of a gap of $H$ in $S^1$.
This implies that $f_*$ 
 is a homeomorphism.

Any deck transformation $\gamma$ permutes the lifts of JSJ
tori and components of $\mt - W$. It sends infinite embedded
paths in the tree $\cT_F$ to infinite paths in
the tree $\cT_{\gamma(F)}$. The tree $\cT_{\gamma(F)}$ is
canonically homeomorphic to the tree $\cT_F$ and this
identification is compatible with the identifications
of $S^1(\gamma(F))$ and $S^1(F)$. It follows that the
homeomorphisms $f_*$ are equivariant.
This finishes the proof of the proposition.
\end{proof}

\begin{remark} Notice that in the case of a non trivial
JSJ decomposition, the construction of the JSJ universal
circle comes with an invariant lamination. 
For definition and properties of invariant laminations
associated to universal circles see \cite[Section 8.2]{Calegari}.
The invariant lamination is obtained from the
tori in the JSJ decomposition, their lifts to $\widetilde M$ and
their ideal points in the ideal circles.
\end{remark}

\subsection{Moving points in the universal circle}

The following property will be important for the proof of Theorem \ref{teo-Minimal}.

\begin{prop}\label{p.movingpoints}
If $\cF$ is a uniform $\R$-covered foliation by hyperbolic leaves and $\xi \in \suniv$ then there is $\gamma \in \pi_1(M)$ such that $\gamma(\xi) \neq \xi$. 
\end{prop}

\begin{proof}
We first treat the case where the JSJ decomposition of $M$ is trivial. If $M$ is Seifert with hyperbolic base, the universal circle is identified with the boundary of the universal cover of the base. The base is 
a hyperbolic surface $S$, maybe with finitely
many orbifold singular points. If $\delta$ is a generator of the center of
$\pi_1(M)$ then $\pi_1(M)/<\delta>$ is isomorphic to a closed 
surface group $\pi_1(S)$ where $S$
may have finitely many orbifold (or cone) points and acts on the boundary
$\partial \widetilde S$. The stabilizer of each point in $\partial
\widetilde S$ is at most infinite cyclic. The deck transformation
$\delta$ acts by the identity on the universal circle of the foliation.
It now follows that the stabilizer of a point of the universal circle
is at most a $\Z \oplus \Z$ subgroup. 
By homological reasons $\Z \oplus \Z$ cannot be the 
fundamental group of an irreducible closed $3$-manifold \cite{Hempel}.
This finishes the proof in the Seifert case.

If $M$ is atoroidal then it is hyperbolic\footnote{This follows from Perelmans' geometrization theorem. We do not need the full force of geometrization here, it is enough to know that atoroidal manifolds have fundamental group which is Gromov hyperbolic \cite{GabaiKa}, see also \cite[Corollary 9.32]{Calegari}.}  and we then assume $\mt = \mathbb{H}^3$. In this case we
show that the stabilizer of $\xi$ is at most infinite cyclic.
Suppose that $\gamma$ is in the stabilizer of $\xi$. Let
$F$ be a leaf of $\widetilde \cF$ and $\xi_F$ be the ideal point
of $S^1(F)$ associated to $\xi$. Thurston \cite{Thurston} proved
that the embedding $F \rightarrow \mt$ extends to a continuous
map $F \cup S^1(F) \rightarrow \mt \cup S^2_{\infty}$ where $S^2_\infty$ is the boundary $\partial_{\infty} \mt = \partial_\infty \mathbb{H}^3$ (cf. \S~\ref{ss.boundaries}). 
Let $p$ be the image of $\xi_F$ under this extended map.
Let $\beta$ be a geodesic ray in $F$ with ideal point $\xi_F$.
Then $\gamma(\beta)$ is a geodesic ray in $\gamma(F)$.
Since $\gamma(\xi) = \xi$, and $\cF$ is uniform, it follows
that $\gamma(\beta)$ has a subray which is a bounded distance
from $\beta$. In $\mt \cup S^2_{\infty}$ the image of $\beta$
limits to $p$. Since $\gamma(\beta)$ has a subray a bounded
distance from a ray of $\beta$, it follows that $\gamma(p)$
is equal to $p$. Hence $\gamma$ is in the stabilizer of $p$.
But it is well known that the stabilizer of a point in $S^2_{\infty}$
is at most cyclic. This finishes the proof in the atoroidal case.

Foliations in manifolds with (virtually) solvable fundamental group are classified and cannot be uniform $\R$-covered with hyperbolic leaves (see \cite{Plante} or \cite[Appendix B]{HP} for the $C^0$-case).
In fact the result does not work for manifolds with (virtually) 
solvable fundamental group. 
So the remaining case to be analyzed in the proof
is is when $M$ has a non-trivial JSJ decomposition in our sense (which excludes being a torus bundle up to a finite cover). 

Now we consider the case that the JSJ decomposition of $M$
is not trivial.

Let $\cT$ be the tree of lifts of the pieces of the JSJ decomposition as in Proposition \ref{p.JSJtree}.
 Fix $F$ a leaf of $\widetilde \cF$.
Recall from the proof of Lemma \ref{denseidealpoints}
that the following holds: for any lift $\tild{M}_{i_0}^{j_0}$ of a piece $M_{i_0}$ of the JSJ decomposition of $M$, it intersects $F$ in a single component.
Let $\xi_F$ be the point of $S^1(F)$ corresponding to $\xi$.
We consider 2 distinct lifts $\tild{M}^j_i$ as follows.
First take an arbitrary $\tild{M}^j_i$ so that 
$\xi_F$ is not an ideal point of $A_1 = \tild{M}^j_i \cap F$.
Now take a second lift $\tild{M}^k_i$ so that
$A_1$ separates $A_2 = \tild{M}^k_i \cap F$ from $\xi_F$ $-$
this means that the closure of $A_1$ in
$F \cup S^1(F)$ separates $\xi_F$ from the closure of $A_2$
in $F \cup S^1(F)$.
We also choose $\tild{M}^k_i$ so that distance in
the JSJ tree from $\tild{M}^k_i$ to $\tild{M}^j_i$ is greater
than $3$.
Let now $\gamma$ be a non trivial deck transformation
that fixes $\tild{M}^k_i$.
By Proposition \ref{p.JSJtree} the diameter of the fixed
point set of $\gamma$ acting on the JSJ tree is less than
or equal to $2$. In particular $\gamma(\tild{M}^j_i)
\not = \tild{M}^j_i$. Since $A_1$ separates $A_2$ from
$\xi_F$, it now follows that $\gamma(\xi)$ is not
equal to $\xi$. 

This concludes the proof of the proposition.
\end{proof}

\begin{remark}
One can give a different proof of Proposition \ref{p.movingpoints} using different machinery that we chose not to present in detail. Indeed, if there is a global fixed point $\xi$ in the universal circle of a uniform $\R$-covered foliation by hyperbolic leaves, then the one-dimensional foliation by geodesics in each leaf landing as a geodesic fan on $\xi$ is equivariant and therefore descends to a one-dimensional foliation (which if chosen to be tangent to a unit vector field defines a flow) in $M$. By an argument in \cite{CalegariLam} (see the proof of \cite[Theorem 5.5.8]{CalegariLam}) this flow
is (topologically) Anosov\footnote{Or at least semiconjugate to it.} for which $\cF$ is the weak stable foliation. This is impossible since the flow would be $\R$-covered and not a suspension (because the center stable foliation is uniform). This flow also does not have periodic orbits freely homotopic to their inverses, because the orbits always point in the
direction of $\xi$. This contradicts  what is proved in \cite{FenleyAnosov,Barbot}. 
\end{remark}

\section{Proof of Theorem \ref{teo-Minimal}}

We fix in $\pi_1(M)$ a finite symmetric set of generators $\cS$ and denote by $|\gamma|$ the word length of $\gamma$ with respect to $\cS$. We will be concerned with sequences going to infinity, so the choice of $\cS$ is irrelevant. 

Theorem \ref{teo-Minimal} concerns the action of $\pi_1(M)$
on the universal circle $\suniv$. The universal circle is
canonically homeomorphic to $S^1(L)$ for any $L$ leaf of $\widetilde \cF$.
By Remark \ref{rem-action}, in order to prove Theorem \ref{teo-Minimal},
it is equivalent to consider the action 
of $\pi_1(M)$ on $S^1(L)$. So fix a leaf $L \in \Ft$ and denote by

$$\rho(\gamma): L \cup S^1(L) \to L \cup S^1(L), \ \ \ \
\rho(\gamma) (x) = \tau_{\gamma L , L} \circ \gamma$$ 

\noindent
This induces an action of $\pi_1(M)$ on $S^1(L)$ 
(but in general does not induce an action on $L$).
Again via the identification with the universal circle $\suniv$ this is exactly the action defined on $\suniv$ in Remark \ref{rem-action}. 
In this way $\rho$ is a group homomorphism from $\pi_1(M)$ into $\mathrm{Homeo}_+(S^1(L))$.

Fix a point $x_0 \in L$. The point $x_0$ allows us to define a \emph{visual measure} (cf. \S \ref{ss.boundaries}) in $S^1(L)$ that we will also fix.  

The first important property is the following: 

\begin{lemma}\label{lema-toinfty} 
Given a compact interval $\cI \subset \cL_{\cF} = \mt/_{\Ft}$ containing $L$ we have that if $\gamma_n \in \pi_1(M)$ satisfies $\gamma_n L \in \cI$ and $|\gamma_n| \to \infty$, then for every $x \in L$ we have $$ d_L (x, \rho(\gamma_n)x) \to \infty. $$
In particular, given $C \subset L$ compact, there is $K>0$ such that if $|\gamma|>K$ and $\gamma L \in \cI $ then $\rho(\gamma)C \cap C = \emptyset$. 
\end{lemma} 

\begin{proof}
Fix a compact fundamental domain $Y$ of $M$ in $\mt$. For a given $R>0$ there is a bounded set $G \subset \mt$ 
which consists of the points $z$ in leaves $F \in \cI$ such that $\tau_{F,L}(z) \in B_R(x)$ where $B_R(x)$ denotes the ball of radius $R$ in $L$. 
The set $G$ is bounded because the quasi-isometry constants of  
$\tau_{F,L}|_L$ depend only on the Hausdorff distance between
$F$ and $L$. Since $F \in \cI$ the Hausdorff distance is bounded.
Now, one can cover $G$ by finitely many fundamental domains, implying that if $\gamma$ verifies that $\gamma L \in \cI$ and $|\gamma|$ is sufficiently large, then $\rho(\gamma) x$ cannot be in $B_R(x)$. This completes the first part of the Lemma. 

For the second statement, notice that estimates are uniform, so by compactness one gets the statement. 
\end{proof} 

This allows us to show the following: 

\begin{lemma}\label{lem-singularvalues}
For every finite interval $\cI \subset \cL_{\cF}$ containing $L$ and $\eps>0$ there is $K>0$ such that if $|\gamma|>K$ and $\gamma L \in \cI$ we have that there are (not necessarily disjoint) intervals $I_\gamma, J_\gamma$ of length (for the visual distance in $S^1(L)$) smaller than $\eps$ and such that $$\rho(\gamma)(S^1\setminus I_\gamma) \subset J_\gamma.$$ 
\end{lemma}

\begin{proof}
Given the finite interval $\cI$ there exists a uniform constant $c>1$ so that for every $F \in \cI$ the map $\tau_{F,L}: F \to L$ is a quasi-isometry with constant $c$. It follows that the image by $\rho(\gamma)$ of a geodesic in $L$ is a $c$-quasigeodesic in $L$ whenever $\gamma L \in \cI$. 
Notice that $\tau_{F,L}|_L$ is not necessarily continuous, so 
$\tau_{F,L}|_L(c)$ not necessarily a continuous curve. But the
quasi-isometry inequalities still hold.

Fix $x_0$ in $L$.
Let $C \subset L$ be a compact set containing $x_0$ with the property that every quasigeodesic in $L$ with constants bounded by $c$ which does not intersect $C$ verifies that its visual measure is smaller than $\eps/2$.

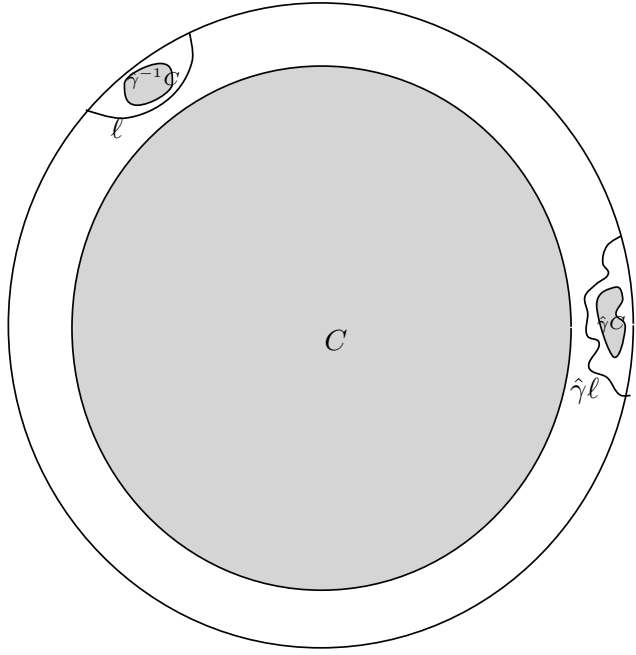
\begin{figure}[ht]
\begin{center}
\includegraphics[scale=0.80]{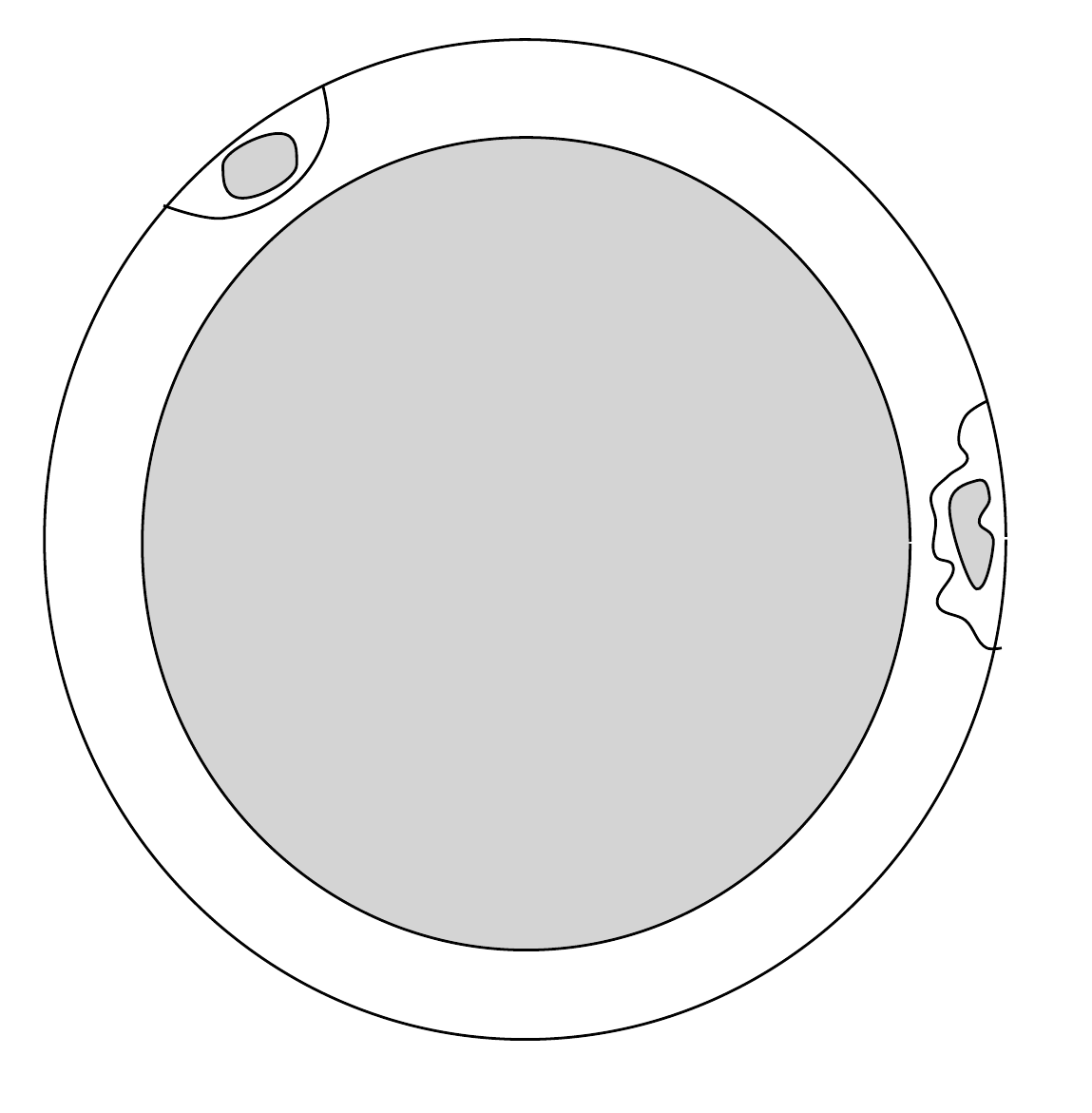}
\begin{picture}(0,0)
\put(-140,130){$C$}
\put(-220,210){$\ell$}
\put(-214,230){{\tiny$\hat \gamma^{-1}C$}}
\put(-38,138){{\tiny $\hat \gamma C$}}
\put(-48,112){$\hat \gamma \ell$}
\end{picture}
\end{center}
\vspace{-0.5cm}
\caption{{\small Depiction of the ingredients of the proof of Lemma \ref{lem-singularvalues}. Here $\hat \gamma:=\rho(\gamma)$.}\label{f.lema52}}
\end{figure}

Now, we can apply Lemma \ref{lema-toinfty} to find $K$ such that if $\gamma$ verifies that $\gamma L \in \cI$ and $|\gamma|>K$ then one has that $\rho(\gamma^{-1}) C \cap C = \emptyset$. By choosing $K$ a bit larger, one can assume that there is a geodesic $\ell$ in $L$
 which separates $\rho(\gamma^{-1}) C$ from $C$ (see figure \ref{f.lema52}). 
This uses that the diameter of $\rho(\gamma^{-1})C \cap C$ is uniformly bounded.
This allows us to define $I_\gamma$ as the (shortest) interval determined by the endpoints of $\ell$ (i.e. the one so that 
$I_{\gamma} \cup l \subset L \cup S^1(L)$ leaves $C$ on the outside) and $J_\gamma$ to the (shortest) interval joining the endpoints of the quasigeodesic $\rho(\gamma)(\ell)$. 

\end{proof}

Now we are in condition to prove minimality of the action:

\begin{prop}\label{p.minimal}
The action of $\pi_1(M)$ on $\suniv$ is minimal. In particular, given $\xi \in S^1(L)$ and an open interval $U \subset S^1(L)$ there exists $\gamma \in \pi_1(M)$ such that $\rho(\gamma)\xi \in U$. 
\end{prop}

\begin{proof}
We first fix an open set $U \subset S^1(L)$.

Fix $T$ a compact fundamental domain of $M$ in $\mt$. Every other fundamental domain will be a translate of $T$ by a deck transformation. Let $D= \mathrm{diam}(T)$, which is also the diameter of any translate of $T$.
Let $\cI \subset \cL_{\cF}$ be a compact interval around $L$ such that the union $\bigcup_{F \in \cI} F$ contains the neighborhood of size $2D$ of the leaf $L$. Notice that this interval can be chosen thanks to the fact that $\cF$ is $\R$-covered and uniform.

Choose points $\xi_1 \neq \xi_2$ in the interior of $U$ and take fundamental domains $T^n_1, T^n_2$ of $M$ in $\mt$ such that they intersect $L$ in points very close to $\xi_1$ and $\xi_2$ respectively, more precisely, such that the intersection $T^n_i \cap L$ is non empty and $T^n_i \cap L \to \xi_i$
in $L \cup S^1(L)$.  See figure \ref{f.minimal}.

Now, we can choose $\gamma_n$ so that $\gamma_n(T_1^n) = T_2^n$. 
Since the diameter of $T^n_i$ is fixed, $T^n_i \cap L \to \xi_i$
and $\xi_1 \not = \xi_2$ then as $n \to \infty$, 
$${\rm for \  any } \ \ \
y^n_1 \in T^n_1, \ y^n_2 \in T^n_2, \ \ \ 
d(y^n_1,y^n_2) \to \infty.$$
\noindent
It follows that $|\gamma_n| \to \infty$.
Also $\gamma_n L \in \cI$ so that Lemma \ref{lem-singularvalues} applies. 
This also uses that $\cF$ is $\R$-covered.
Let $I_{\gamma_n}, J_{\gamma_n}$ be the intervals provided
by Lemma \ref{lem-singularvalues}.
We choose $\eps>0$ small so that the $2\eps$-neighborhood of both $\xi_1$ and $\xi_2$ in $S^1(L)$ is contained in $U$.  

By contradiction, we assume that there are arbitrarily large $n$ so
 that  neither $I_{\gamma_n}$ nor $J_{\gamma_n}$ are contained in $U$. Take $\hat U$ the subinterval of $U$ obtained by removing from
$U$ the $\epsilon$ neighborhoods of the endpoints
(i.e. if $U = (a,b)$ we consider $\hat U=(a+\eps, b-\eps)$). 
Notice that $\xi_1, \xi_2$ are in $\hat U$.
The choice of $I_{\gamma_n}$ and $J_{\gamma_n}$ implies that they are both disjoint from $\hat U$ and therefore $\gamma_n \hat U \cap \hat U = \emptyset$.  This will be a contradiction as follows:
take $c_n$ a geodesic in $L$ intersecting $T^n_1 \cap L$ whose endpoints are close to $\xi_1$ and contained in $\hat U$. 
Then the image by $\rho(\gamma_n)$ of $c_n$ is a uniform quasigeodesic,
because $\gamma_n L$ is in a compact interval $\cI$ in the leaf space.
This uniform quasigeodesic intersects 
$T^n_2$ and therefore has at least one endpoint 
in a neighborhood of $\xi_2$ if $n$ is large enough
(note that we cannot ensure that  both endpoints of $\rho(\gamma_n)(c_n)$ are contained in $\hat U$). This implies that $\gamma_n \hat U \cap \hat U \neq \emptyset$ which is a contradiction.

%

Therefore up to a subsquence and replacing $U$ by a slightly smaller
open set, it follows that either $I_{\gamma_n}$ or $J_{\gamma_n}$
is contained in $U$.
Up to taking $\gamma_n^{-1}$ we can assume that $J_{\gamma_n} \subset U$.

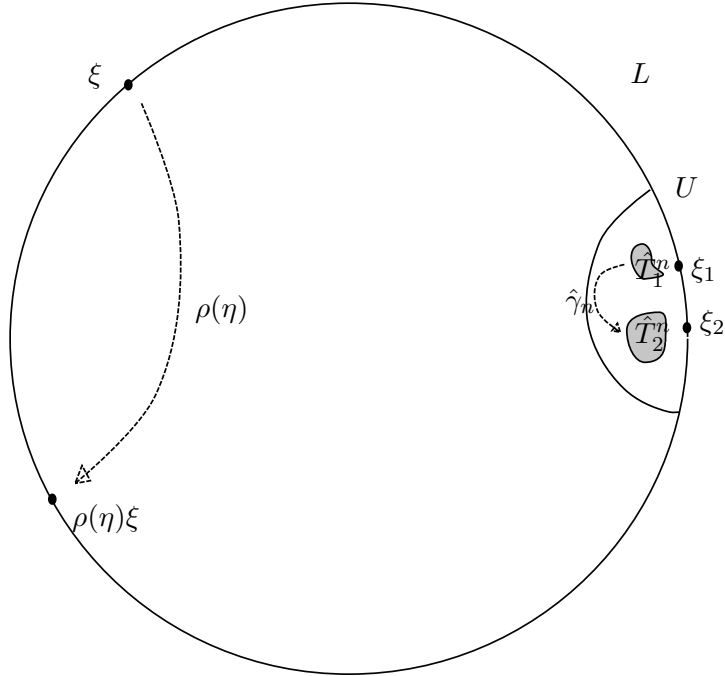
\begin{figure}[ht]
\begin{center}
\includegraphics[scale=0.80]{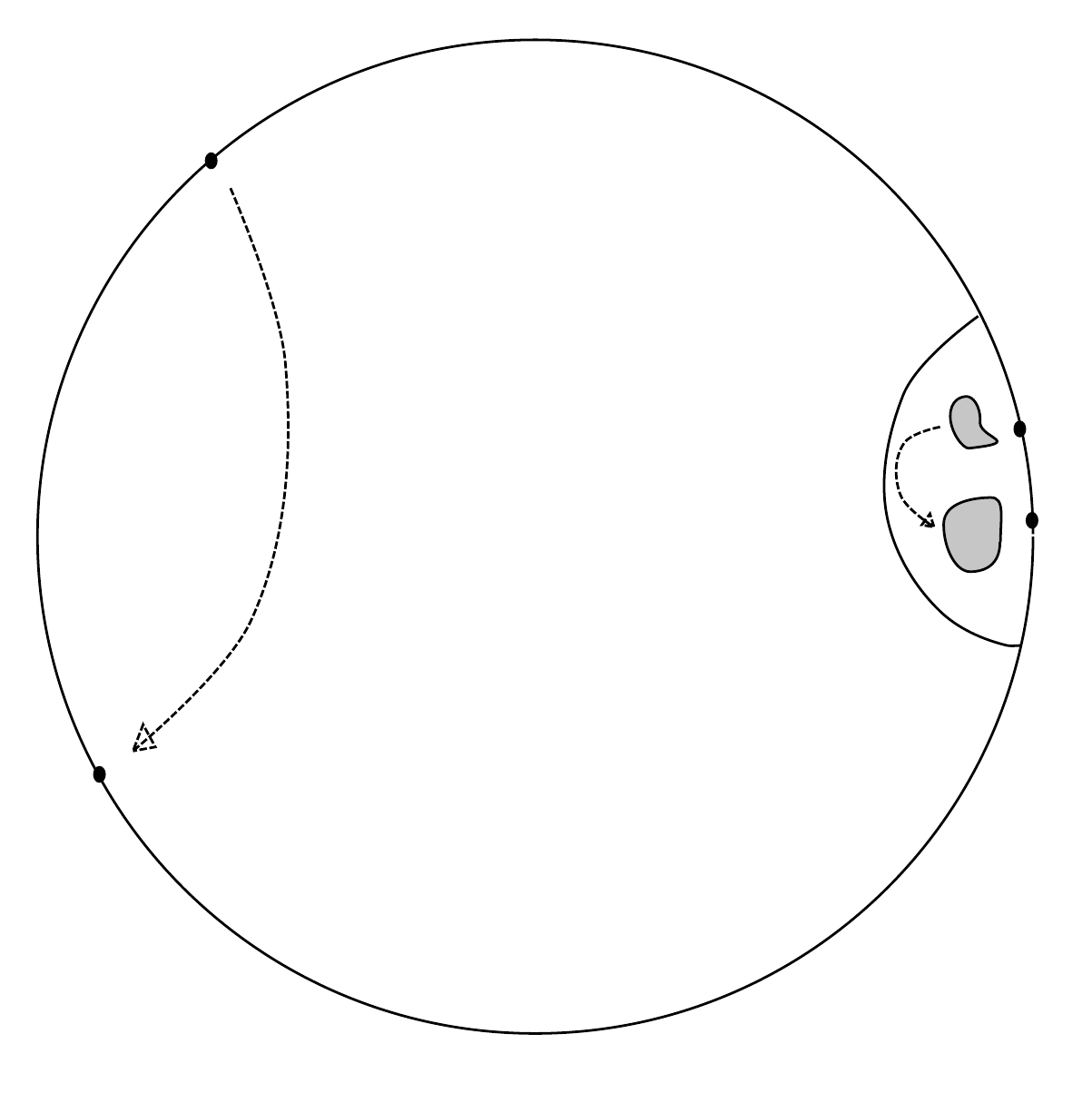}
\begin{picture}(0,0)
\put(-40,238){$L$}
\put(-203,150){$\rho(\eta)$}
\put(-243,238){$\xi$}
\put(-249,71){$\rho(\eta) \xi$}
\put(-27,195){ $U$}
\put(-21,165){ $\xi_1$}
\put(-14,145){$\xi_2$}
\put(-42, 164){{ $\hat T_1^n$}}
\put(-42, 140){{ $\hat T_2^n$}}
\put(-68, 153){{ $\hat \gamma_n$}}
\end{picture}
\end{center}
\vspace{-0.5cm}
\caption{{\small Depiction of the ingredients of the proof of Proposition \ref{p.minimal}. Here $\hat \gamma_n = \rho(\gamma_n)$ and $\hat T_i^n = T_i^n \cap L$.}\label{f.minimal}}
\end{figure}

We now choose an arbitrary point $\xi$ in $S^1(L)$. Pick $\eta \in \pi_1(M)$ so that $\rho(\eta)\xi \neq \xi$ (cf. Proposition \ref{p.movingpoints}). 
If necessary choose $\eps$ smaller so that the distance in $S^1(L)$
from $\xi$ to $\rho(\eta)\xi$ is bigger than $10 \eps$.

Assume first that $\xi \notin I_{\gamma_n}$ for arbitrarily large $n$. In this case, one concludes since $\rho(\gamma_n) \xi \in J_{\gamma_n} \subset U$ as desired. If $\xi \in I_{\gamma_n}$ for all large $n$, then by the choice
of $\eps$ it follows that 
$\rho(\eta)\xi \notin I_{\gamma_n}$ for large enough $n$. This implies that $\rho(\gamma_n \eta)\xi \in J_{\gamma_n} \subset U$ completing the 
proof of the proposition.
\end{proof}

We devote the rest of the section to the proof of transitivity 
of the action on pairs of points. First, we show that we can find attractor/repeller configurations in any pair of open sets. 

\begin{lemma}\label{lema-ar}
For every $U,V$ open intervals in $S^1(L)$ there is $\gamma \in \pi_1(M)$ such that $\rho(\gamma)(S^1(L) \setminus U) \subset V$. 
\end{lemma}

\begin{proof}
Consider a sufficiently large compact interval $\cI \subset \cL_\cF$ as in the proof of Proposition \ref{p.minimal} so that the union of its leaves contains a neighborhood of size larger than the diameter of a fundamental domain around $L$.  

As in the proof of Proposition \ref{p.minimal}, it is possible to construct a sequence $\gamma_n \in \pi_1(M)$ such that $|\gamma_n| \to \infty$ and such that the neighborhoods $I_{\gamma_n}$ and $J_{\gamma_n}$ verify (up to taking a subsequence) that $I_{\gamma_n} \to \xi_1$ and $J_{\gamma_n} \to \xi_2$ where it could be that $\xi_1 = \xi_2$. This is just taking very large elements that move a fundamental domain intersecting $L$ into other fundamental domain intersecting $L$ and applying Lemma \ref{lem-singularvalues}. 

Now, using Proposition \ref{p.minimal} we choose $\eta_1$ and $\eta_2$ in
$\pi_1(M)$ satisfying $\rho(\eta_1)(\xi_1) \in U$ and 
$\rho(\eta_2)(\xi_2) \in V$. It follows that for sufficiently large $n$ the deck transformation $\beta_n= \eta_2 \circ \gamma_n \circ \eta_1^{-1}$ verifies that  $\rho(\beta_n) (S^1\setminus U) \subset V$. 

To see this, notice that $\rho(\eta_1^{-1}) (U)$ contains $I_{\gamma_n}$ for suficiently large $n$ because $\rho(\eta_1)(\xi_1) \in U$. Similarly, if $n$ is large enough, then $\rho(\eta_2)(J_{\gamma_n})$ is contained in $V$. Since $\rho(\gamma_n)(S^1(L) \setminus I_{\gamma_n}) \subset J_{\gamma_n}$, this completes the proof.  
\end{proof}

To complete the proof of Theorem \ref{teo-Minimal} it is enough to show:

\begin{prop}\label{prop-trans}
Given open intervals $U_1, V_1 \subset S^1(L)$ and $U_2, V_2 \subset S^1(L)$ there exists $\hat \gamma \in \pi_1(M)$ such that $\rho(\hat \gamma) U_1 \cap U_2 \neq \emptyset$ and $\rho(\hat \gamma) V_1 \cap V_2 \neq \emptyset$. In particular, there exists a pair $\xi_1 \neq \xi_2 \in \suniv$ whose $\pi_1(M)$-orbit is dense in $\suniv \times \suniv \setminus \{\text{diagonal}\}$. 
\end{prop} 

\begin{proof}
By reducing the intervals we can assume without loss of generality that the four intervals $U_1, U_2, V_1, V_2$ are disjoint. 

Apply Lemma \ref{lema-ar} to find deck transformations $\gamma$ and $\eta$ which verify that $\rho(\gamma )(S^1(L) \setminus U_1) \subset V_2$ and $\rho(\eta) (S^1(L) \setminus \rho(\gamma) V_1) \subset U_2$. 

Now, the transformation $\hat \gamma=\eta \gamma$ is the desired one. Indeed, 
$$\rho(\gamma)U_1 \cap \rho(\gamma)V_1 = \emptyset, 
\ \ {\rm or } \ \ 
\rho(\gamma)U_1 \subset S^1(L) \setminus \rho(\gamma)V_1,$$

\noindent
which implies that 
$\rho(\eta \gamma) U_1 \subset U_2$.
 In addition 


$$\rho(\eta \gamma)V_1 = \rho(\eta)\rho(\gamma)V_1
\supset S^1(L) \setminus U_2 \supset V_2.$$

The existence of dense orbits is now standard. Indeed, pick a countable basis $\{U_n\}$ of intervals generating the topology of $S^1(L)$. The set $A_{n,m}$ of pairs of different points $\xi_1, \xi_2$ such that there exists $\gamma \in \pi_1(M)$ such that $\rho(\gamma)\xi_1 \in U_n$ and $\rho(\gamma) \xi_2 \in U_m$ is clearly open and it is dense because of what we just proved. Then, the intersection $\bigcap_{n,m} A_{n,m}$ is a residual subset by Baire's category theorem and the orbit of points in $A_{n,m}$ is always dense in $\suniv \times \suniv$. 
\end{proof}

\section{Branching foliations}
In this section we just point out that all our results work in the setting of branching foliations as they appear in the study of partially hyperbolic dynamics. These objects were introduced by Burago-Ivanov \cite{BI}. We give here a definition that excludes a priori the existence of Reeb component like objects. 

A \emph{branching foliation} $\cF_{bran}$ in a 3-manifold $M$ is a collection of immersed surfaces (tangent to a continuous distribution) called leaves with the following properties. If $\tild{\cF}_{bran}$ is the lift of the collection to $\mt$ then:

\begin{itemize}
\item Each leaf $L$ of $\tild{\cF}_{bran}$ is a properly embedded plane in $\mt$ and separates $\mt$ in two open regions $L^{\oplus}$ and $L^{\ominus}$. Denote $L^+= L \cup L^{\oplus}$ and $L^- = L \cup L^{\ominus}$.  
\item Every point in $\mt$ belongs to at least one leaf $L \in \tild{\cF}_{bran}$. 
\item The leaves do not topologically cross. That is, given two leaves $L$ and $F$ of $\tild{\cF}_{bran}$ we have that $F \subset L^+$ or $F \subset L^-$. 
\item Given a sequence of points $x_n \to x \in \mt$ and leaves $L_n$ with $x_n \in L_n$ it follows that through $x$ there is a leaf $L \in \tild{\cF}_{bran}$ which is the uniform limit in compact parts of $L_n$. 
\end{itemize}

In \cite[\S 3]{BFFP2} (see also \cite[\S 3]{BFP}) a careful study of the properties of these objects is performed, including a study of the \emph{leaf space} associated to such a branching foliation. In particular, it makes perfect sense to talk about uniform branching foliations and $\R$-covered ones. Moreover, in the
partially hyperbolic setting there exists foliations in $M$ that approach 
the center stable and center unstable branching foliations.
In this setting this can be used to have in general situations a metric in $M$ which gives curvature arbitrarily close to 
$-1$ to all leaves of $\cF$. In this setting, one can define a universal circle as one does for general foliations. 

All arguments performed in this note thus hold for branching foliations. We state the result in this context for future use and explain how it can be deduced from the results proved in this paper. (Note that we could have performed our arguments directly in the branching foliation setting, but we decided to work out the foliation case first since we believe that many people may only be interested by the true foliation case.) 

\begin{teo}
Let $\cF$ be a uniform branching foliation. Then, it is $\R$-covered. Moreover, if $M$ admits a metric making every leaf negatively curved, then the action of $\pi_1(M)$ is minimal in the universal circle $\suniv$ and moreover it acts transitively in pairs of points of $\suniv$. 
\end{teo}

\begin{proof} Note that \cite[Theorem 7.2]{BI} (see also \cite[Theorem 3.3]{BFP}) shows that a transversely oriented 
branching foliation $\cF$ can be approximated by a true foliation $\cF_\eps$ together with a continuous and surjective map $h_\eps: M \to M$, which
is $\eps$ close to the identity, and is a local diffeomorphism from $L \in \cF_\eps$ to $h_\eps(L) \in \cF$. 
In particular the $\eps$-close property shows that the approximating
foliation is also uniform.
Moreover, the map $h_\eps$ lifted to $\mt$
induces a homeomorphism between the leaf spaces (see \cite[Theorem 3.3 (ii)]{BFP}). In this way, we can associate to a uniform branching foliation $\cF$ another uniform (non branching) foliation $\cF_\eps$ and thus apply Theorem \ref{teo-Rcovered} to $\cF_\eps$ to obtain the same statement for $\cF$. 

To show the minimality, we can either use the same arguments as in the previous section (which work without modifications), or alternatively, we can also use $\cF_\eps$ and note that $h_\eps$ induces a conjugacy between the actions in the universal circle. 

There is always a double cover of $M$ so that $\cF$ lifts to a transversely
oriented branching foliation, so the result follows.
\end{proof}

\medskip 
\medskip

\emph{Acknowledgements:} S.F was partially supported by Simons Foundation grant number 637554. R. P. was partially supported by CSIC 618, FCE-1-2017-1-135352. This work was completed while R.P. was a Von Neumann fellow at IAS, funded by the Minerva Research Foundation Membership Fund and NSF DMS-1638352. The authors thank the referee for a very careful reading that helped improve the paper including spotting some inaccuracies.


\begin{thebibliography}{2}


\bibitem[BFFP]{BFFP2} T.~Barthelm\'e, S.~Fenley, S. Frankel, R. Potrie,  Partially hyperbolic diffeomorphisms homotopic to the identity in dimension 3, Part II: branching foliations, \emph{Preprint} arXiv:2008.04871v1. 

\bibitem[BFP]{BFP} T.~Barthelm\'e, S.~Fenley, R. Potrie, Collapsed Anosov flows and self orbit equivalences, \emph{Preprint} arXiv:2008.06547


\bibitem[Bar]{Barbot} T. Barbot, Caract\'erisation des flots d'Anosov en dimension 3 par leurs feuilletages faibles, \emph{ Ergodic Theory Dynam. Systems} {\bf 15} (1995), no. 2, 247--270.

\bibitem[Brit]{Brit} M. Brittenham, Essential laminations in Seifert fibered spaces, \emph{Topology} {\bf 32} (1993) 61-85.



\bibitem[BI]{BI} D. Burago, S. Ivanov, Partially hyperbolic diffeomorphisms of 3-manifolds with abelian fundamental groups, \emph{J. Mod. Dyn}.{\bf 2} (2008), no. 4, 541--580. 

\bibitem[Ca$_1$]{CalegariPA} D.~Calegari, The geometry of ${\bf R}$-covered foliations. \emph{Geom. Topol.} {\bf 4} (2000), 457--515.

\bibitem[Ca$_2$]{Calegari0} D.~Calegari, Leafwise smoothing laminations, \emph{Algebr. Geom. Topol.} {\bf 1} (2001), 579--585. 

\bibitem[Ca$_3$]{CalegariLam}  D.~Calegari, Promoting essential laminations. \emph{Invent. Math.} {\bf 166} (2006), no. 3, 583--643.

\bibitem[Ca$_4$]{Calegari} D.~Calegari,  \emph{Foliations and the geometry of 3-manifolds.} Oxford Mathematical Monographs. Oxford University Press, Oxford, (2007). xiv+363 pp. ISBN: 978-0-19-857008-0

\bibitem[CD]{CD} D. Calegari, N. Dunfield,  Laminations and groups of homeomorphisms of the circle. \emph{Invent. Math.} {\bf 152} (2003), no. 1, 149--204. 

\bibitem[Can]{Candel} A.~Candel, Uniformization of surface laminations, \emph{ Ann. Sci. \'{E}cole Norm. Sup. } {\bf 26} (1993) 489-516.

\bibitem[CC]{CanCon} A. Candel, L. Conlon, \emph{Foliations I and II,} Graduate Studies in Mathematics, 60. American Mathematical Society, Providence, RI,  (2000) xiv+402 pp. ISBN: 0-8218-0809-5 and (2003) xiv+545 pp. ISBN: 0-8218-0881-8

\bibitem[CKR]{CKR} V. Colin, W. Kazez, R. Roberts, Taut foliations. \emph{Comm. Anal. Geom.} {\bf 27} (2019), no. 2, 357--375. 

\bibitem[Fen$_1$]{FenleyQI} S.~Fenley, Quasi-isometric foliations,
	\emph{Topology}, {\bf 31} (1992) 667-676.

\bibitem[Fen$_2$]{FenleyAnosov} S.~Fenley, Anosov flows in $3$-manifolds. \emph{Ann. of Math. (2)} {\bf 139} (1994), no. 1, 79--115.

\bibitem[Fen$_3$]{Fen2002} S.~Fenley, Foliations, topology and geometry of 3-manifolds: $\mathbb{R}$-covered foliations and transverse pseudo-Anosov flows. \emph{Comment. Math. Helv.} {\bf 77} (2002), no. 3, 415--490.

\bibitem[FP$_1$]{FP-acc} S. Fenley, R. Potrie, Ergodicity of partially hyperbolic diffeomorphisms in hyperbolic 3-manifolds, \emph{Preprint} arXiv:1809.02284v2

\bibitem[FP$_2$]{FP-2} S. Fenley, R. Potrie, Partial hyperbolicity and pseudo-Anosov dynamics. \emph{Preprint} arXiv:2102.02156. 

\bibitem[Fra]{Frankel} S. Frankel, Coarse hyperbolicity and closed orbits for quasigeodesic flows. \emph{Ann. of Math.} (2) {\bf 188} (2018), no. 1, 1--48.

\bibitem[Gab]{Gabai} D. Gabai, Combinatorial volume preserving flows and taut foliations, \emph{Comment. Math. Helv.} {\bf 75} (2000), no. 1, 109--124. 

\bibitem[GabKa]{GabaiKa} D.Gabai, W. Kazez, Group negative curvature for 3-manifolds with genuine laminations. \emph{Geom. Topol.} {\bf 2} (1998), 65--77. 

\bibitem[Gr]{Gromov}  M. Gromov, Hyperbolic groups. Essays in group theory, 75--263, Math. Sci. Res. Inst. Publ., 8, Springer, New York, 1987.

\bibitem[HP]{HP} A. Hammerlindl, R. Potrie, Pointwise partial hyperbolicity in three-dimensional nilmanifolds.\emph{ J. Lond. Math. Soc. }(2){\bf 89} (2014), no. 3, 853--875. 

\bibitem[Hat]{Hatcher} A. Hatcher, Notes on basic 3-manifold topology, \emph{Lecture notes available on the author's web page.}

\bibitem[Hem]{Hempel} J. Hempel, \emph{3-Manifolds.} Ann. of Math. Studies, No. 86. Princeton University Press, Princeton, N. J.; University of Tokyo Press, Tokyo, 1976. xii+195 pp. 

\bibitem[KL]{KL} M. Kapovich, B. Leeb, 3-manifold groups and nonpositive curvature, \emph{Geom. Funct. Anal} {\bf 8} (1998) 841--852. 

\bibitem[Led]{Ledrappier} F. Ledrappier,
Structure au bord des vari\'et\'es \`a courbure n\'egative, \emph{ S\'eminaire de Th\'eorie Spectrale et G\'eom\'etrie,} No. 13, Ann\'ee 1994-1995, 97--122, 
S\'emin. Th\'eor. Spectr. G\'eom., 13, Univ. Grenoble I, Saint-Martin-d'H\`eres, 1995. 

\bibitem[Ng]{Ng} H. T. Nguyen, Distorsion of surfaces in 3-manifolds, \emph{J. of Topol.} {\bf 12} (2019) 1115--1145. 

\bibitem[Nov]{Nov} S.P.~Novikov, Topology of foliations, \emph{Trans. Moscow Math. Soc} {\bf 14} (1963) 268--305.

\bibitem[Pal]{Palmeira} C. Palmeira, Open manifolds foliated by planes, \emph{Ann. Math.} (2) {\bf 107} (1978), no. 1, 109--131.

\bibitem[Pla]{Plante} J. F. Plante, Foliations of 3-manifolds with solvable fundamental group, \emph{Invent. Math.} {\bf 51} (1979) 219--230.

\bibitem[Th]{Thurston} W.~Thurston, Three-manifolds, Foliations and Circles, I, \emph{Preprint} arXiv:math/9712268

\end{thebibliography}
\end{document}